\documentclass[10pt, leqno]{article}
\usepackage{geometry}           
\usepackage{amsmath, amsthm, amssymb, titlesec, titletoc}

\newtheorem{theorem}{Theorem}[section]
\newtheorem{corollary}[theorem]{Corollary}
\newtheorem{lemma}[theorem]{Lemma}
\newtheorem{proposition}[theorem]{Proposition}
\theoremstyle{definition}
\newtheorem{definition}[theorem]{Definition}
\newtheorem{remark}[theorem]{Remark}

\numberwithin{equation}{section}

\titleformat{\section}[block]{\scshape\filcenter}{\thesection.}{3pt}{}
\titleformat{\subsection}[block]{\scshape\filcenter}{\thesubsection.}{3pt}{}
\titlecontents{section}[0pt]{}{\thecontentslabel. \,}{}{\hfill\contentspage}

\title{\large{\textbf{A CANONICAL NEIGHBORHOOD THEOREM FOR MEAN CURVATURE FLOW IN HIGHER CODIMENSION}}}
\author{\textsc{\small KEATON NAFF}}
\date{}

\begin{document}
\maketitle

\begin{abstract}In dimensions $n \geq 5$, we prove a canonical neighborhood theorem for the mean curvature flow of compact $n$-dimensional submanifolds in $\mathbb{R}^N$ satisfying a pinching condition $|A|^2 < c|H|^2$ for $c = \min \{ \frac{3(n+1)}{2n(n+2)},\frac{1}{n-2}\}.$\end{abstract}

\maketitle

\section{Introduction}

In this paper, we continue our study of singularity formation of the mean curvature flow in higher codimension for initial data that satisfies a natural curvature pinching condition. We consider compact $n$-dimensional solutions of mean curvature flow, $F : M \times [0, T) \to \mathbb{R}^N$, for which the initial immersion satisfies $|A|^2 < c |H|^2$. In \cite{AB10}, Andrews and Baker showed this condition is preserved by the mean curvature flow (in any codimension) when $c \leq \frac{4}{3n}$. Moreover, for $c \leq \min\{\frac{4}{3n}, \frac{1}{n-1}\}$ they proved a suitable extension (to higher codimension) of Huisken's classical result \cite{Hui84} on the flow of closed convex hypersurfaces into spheres. Recently in \cite{Ngu18}, Nguyen proved cylindrical and pointwise derivative estimates when $c \leq\min\{\frac{4}{3n}, \frac{1}{n-2}\}$, thereby similarly extending two of the very important a priori estimates in the impactful works of Huisken and Sinestrari \cite{HS99, HS09} on the flow of closed, two-convex hypersurfaces. Following the work of Nguyen, in recent papers \cite{Naf19a, Naf19b}, for $c \leq \min \{\frac{3(n+1)}{2n(n+2)},\frac{1}{n-2}\}$, we have shown that the only blow-up models at the first singular time for these pinched solutions of mean curvature flow are the codimension one shrinking round spheres, shrinking round cylinders, and translating bowl solitons. This is an extension of the classification result of Brendle and Choi \cite{BC20}. 

The purpose of this paper is to upgrade our description of the infinitesimal scale at spacetime points of infinite curvature to a description of small scales around spacetime points of high curvature. Since Perelman's work \cite{Per02a} on the Ricci flow in three dimensions, such results are nowadays known as canonical neighborhood theorems. In the mean curvature flow, Huisken and Sinestrari essentially proved a canonical neighborhood theorem through their Neck Detection Lemma and Neck Continuation Theorem (Lemma 7.4 and Theorem 8.2 in \cite{HS09}). The first of these two results has been extended to higher codimension by Nguyen in \cite{Ngu18}. It is an interesting problem to prove a version of the Neck Continuation Theorem in higher codimension. Huisken and Sinestrari's proof of this theorem makes significant use of the notion of convexity, which is absent in higher codimension, so some new ideas will be needed. At the same time as this work was completed, Nguyen has published a preprint where he has addressed this problem \cite{Ngu20} (as well as developed a surgery procedure). We have a different approach here and prove the following theorem, which is much closer in spirit to Perelman's result. 

\begin{theorem}\label{canonical-nbhd-thm}
Suppose $F_0 : M \to \mathbb{R}^N$ is an immersion of a closed manifold of dimension $n \geq 5$ satisfying $|A|^2 < \tilde{c}_2 |H|^2$, where $\tilde{c}_2 := \min \{ \frac{3(n+1)}{2n(n+2)},\frac{1}{n-2}\}$. There exist constants $\tilde \varepsilon$ and $\tilde K$, depending upon $F_0$, with the following property. Let $F : M \times [0, T) \to \mathbb{R}^{N}$ denote the solution of mean curvature flow with initial immersion given by $F_0$. Given $\varepsilon_0 \in (0, \tilde \varepsilon)$ and $K_0 \in (\tilde K, \infty)$, there exists a positive number $\hat r > 0$, depending upon $F_0$, $\varepsilon_0$, and $K_0$, with the following property. If $(p_0, t_0)$ is a spacetime point such that $Q_0:= |H|(p_0, t_0) \geq \hat r^{-1}$, then the solution is $\varepsilon_0$-close in the intrinsic parabolic neighborhood $B_{g(t_0)}(p_0, Q_0^{-1} K_0) \times [t_0 - K_0 Q_0^{-2}, t_0]$ to an ancient model solution.  
\end{theorem}

In fact, one can show that the constants $\tilde \varepsilon$ and $\tilde K$ depend only upon the dimension (see Remark \ref{constants}). Here, an \textit{ancient model solution} is an $n$-dimensional solution of mean curvature flow in $\mathbb{R}^N$, which is ancient, nonflat, complete, codimension one (i.e. lying in an $(n+1)$-dimensional hyperplane), uniformly two-convex, and noncollapsed (see Definition \ref{ancient_model}). The meaning of \textit{$\varepsilon_0$-close} is contained in Definition \ref{varepsilon_close}. Roughly, it means that the family of immersions $F(\cdot, t)$ is close to the family of immersions of a model after reparametrization.  

To prove this theorem, we will follow the original strategy of Perelman \cite{Per02a}. In particular, we will adapt a variation of Perelman's proof given by Brendle in \cite{Bre19} to our setting. In Section \ref{prelims}, we establish definitions and results we will use in the proof of Theorem \ref{canonical-nbhd-thm}. In Section \ref{canon}, we give the proof of Theorem \ref{canonical-nbhd-thm}. In the Appendix, we discuss a local compactness property for solutions of mean curvature flow that is used in our proof.  \\

\noindent \textbf{Acknowledgements:} I would like to thank my advisor, Prof. Simon Brendle, for many inspiring discussions and for his direction in this work.

\section{Preliminaries}\label{prelims}

Henceforth, we will let $H$ denote the scalar mean curvature and $\vec H$ denote mean curvature vector. In higher codimension, this means $H = |\vec H|$. Since we assume $|A|^2 < c H^2 \implies H > 0$, we may define a $(0,2)$-tensor $h_{ij} := \langle A_{ij}, H^{-1} \vec H \rangle$, where $A$ is the full vector-valued version of second fundamental form. In codimension one, $h$ is just the usual scalar-valued version of the second fundamental form. We will adopt the notations $P(p, t, r, \theta) := B_{g(t)}(p, r) \times [t - \theta, t]$ and 
\[
\hat P(p, t, r, \theta) := P(p, t, H(p, t)^{-1}r, H(p, t)^{-2} \theta),
\]
 as in \cite{HS09}, to denote intrinsic parabolic neighborhoods.

We will use the following definitions for $\varepsilon$-necks and $\varepsilon$-caps. 

\begin{definition}\label{necks_caps}
Let $\varepsilon > 0$ be a small positive constant and $F : M \to \mathbb{R}^N$ an isometric immersion of a complete Riemannian manifold. Let $\bar g$ denote the standard metric on the round cylinder $S^{n-1} \times \mathbb{R}$ of radius $1$ (or, equivalently, of constant scalar curvature $(n-1)(n-2)$). 
\begin{itemize}
\item An $\varepsilon$-neck is a compact region $N \subset M$ for which there exists a diffeomorphism $\phi : S^{n-1} \times [-\varepsilon^{-1}, \varepsilon^{-1}] \to N$, a positive constant $r > 0$, and an isometric embedding $\bar F : S^{n-1} \times [-\varepsilon^{-1}, \varepsilon^{-1}] \to \mathbb{R}^N$ (with respect to $\bar g$) such that the immersion $r^{-1}(F \circ \phi)$ is $\varepsilon$-close in $C^{[1/\varepsilon]}$ on $S^{n-1} \times [-\varepsilon^{-1}, \varepsilon^{-1}]$ to the embedding $\bar F$ with respect to the metric $\bar g$. The constant $r$ is called the radius of the neck $N$. For any $z \in [-\varepsilon^{-1}, \varepsilon^{-1}]$, we call $\phi(S^{n-1} \times \{z\}) \subset N$ a cross-sectional sphere of the neck. 
\item We say a point $p_0 \in M$ lies at the center of an $\varepsilon$-neck if $p_0$ lies on the central cross-sectional sphere, $\phi(S^{n-1} \times \{0\})$, of a neck of radius $\frac{n-1}{H(p_0)}$. 
\item An $\varepsilon$-cap is a compact region $D \subset M$ diffeomorphic to a closed $n$-dimensional ball with the property that $\partial D$ is the central cross-sectional sphere of an $\varepsilon$-neck.
\end{itemize}
\end{definition} 

Recently, there has been significant progress in the classification of ancient solutions of mean curvature flow that are two-convex and noncollapsed. By the works of Brendle and Choi \cite{BC19, BC20} and of Angenent, Daskalopoulos, and \v Se\v sum \cite{ADS19,ADS20}, an $n$-dimensional, ancient, uniformly two-convex, noncollapsed and nonflat solution of mean curvature flow in $\mathbb{R}^{n+1}$ is either a family of shrinking round spheres, a family of shrinking round cylinders, a translating bowl soliton, or an ancient oval. Only the first three can arise as blow-up limits at the first singular time. It is still an open problem to determine whether an ancient oval can occur as a singularity model at subsequent singular times (see \cite{CHH21} for more). In any case, for the canonical neighborhood theorem, it is important to include the ancient ovals in our class of model solutions since it is possible for regions of high curvature to be modeled on domains within ancient ovals. 

Similar progress has also occurred for the Ricci flow in dimension three. In particular, a classification of ancient $\kappa$-solutions in dimension three has been established in \cite{Bre20, BDS19}. In the Ricci flow, Perelman did not use an explicit classification of model solutions to prove his canonical neighborhood theorem. However, as above, a qualitative description of the ancient oval solution was still necessary. Perelman \cite{Per02b} gave the first construction of ancient ovals for the Ricci flow in three dimensions (White \cite{Whi03} gave the analogous construction for the mean curvature flow). We will not need an explicit classification of ancient model solutions either. So that our results do not depend upon a classification, we opt to use the following definition for ancient model solutions. 

\begin{definition}\label{ancient_model}
An ancient model solution is an $n$-dimensional ancient, nonflat, complete, connected, codimension one solution of mean curvature flow in $\mathbb{R}^N$ that is uniformly two-convex and noncollapsed. 
\end{definition}
Note that since the model flow is contained in an $(n+1)$-dimensional hyperplane, here noncollapsing and two-convexity are meant with respect to this hyperplane.

By Theorem 1.11 in Haslhofer-Kleiner \cite{HK17a}, any ancient, noncollapsed, mean-convex solution of mean curvature flow is automatically weakly convex. Moreover, by Theorem 1.8 in \cite{HK17a}, given a noncollapsing constant $\alpha > 0$, there exists constants $\gamma_1:= \gamma_1(n, \alpha)$ and $\gamma_2:= \gamma_2(n, \alpha)$ with the property that any ancient, $\alpha$-noncollapsed, mean-convex solution of mean curvature flow satisfies the pointwise derivative estimates $|\nabla A| \leq \gamma_1 H^2$ and $|\nabla^2 A| \leq \gamma_2H^3$. In \cite{Naf19b}, we proved two structure theorems for weakly convex, uniformly two-convex, ancient solutions satisfying the derivative estimates $|\nabla A| \leq \gamma_1 H^2$ and $|\nabla^2 A| \leq \gamma_2 H^3$. Our work also shows that these two derivative estimates together with convexity and uniform two-convexity imply noncollapsing (a sort of converse to the Haslhofer-Kleiner result). The following proposition is a straightforward corollary of ``tube and cap'' structure results in \cite{Naf19b}. It is also a consequence of the works \cite{ADS20,BC20} as well as \cite{HK17b}.

\begin{proposition}\label{model_nbhds}
Given $\varepsilon > 0$ and $\alpha > 0$, there exist positive constants $C_1 := C_1(n, \alpha, \varepsilon)$ and $C_2 := C_2(n, \alpha, \varepsilon)$ with the following property. Assume $\bar F : \bar M \times (-\infty, T) \to \mathbb{R}^N$ is an ancient model solution which is $\alpha$-noncollapsed. Then for each space-time point $(p_0, t_0)$, there exists a closed neighborhood $B \subset \bar M$ containing $p_0$ such that $B_{g(t_0)}(p_0,C_1^{-1} H(p_0, t_0)^{-1}) \subset B \subset B_{g(t_0)}(p_0,C_1 H(p_0, t_0)^{-1})$ and $C_2^{-1} H(p_0, t_0) \leq H(p, t_0) \leq  C_2 H(p_0, t_0)$ for every $p \in B$. Moreover, the neighborhood $B$ is either an $\varepsilon$-neck, an $\varepsilon$-cap, or a closed manifold diffeomorphic to $S^n$. 
\end{proposition}

\begin{remark}\label{constants}
The constants $\tilde K$ and $\tilde \varepsilon$  from Theorem \ref{canonical-nbhd-thm} will depend upon $n, C_1$, and $C_2$. The initial immersion $F_0$ determines scale-invariant derivative estimate bounds $\gamma_1, \gamma_2$ satisfied by the singularity models (as in \cite{Ngu18}, following ideas of \cite{HS09}). These constants $\gamma_1$ and $\gamma_2$ in turn determine the noncollapsing constant $\alpha$ satisfied by the models (by \cite{Naf19b}). In this way, $\tilde{K}$ and $\tilde{\varepsilon}$ will depend upon $n$ and $F_0$. Note, however, the classification of ancient model solutions shows these models are universally noncollapsed for some $\alpha = \alpha(n)$. By relying upon the classification, we could remove the dependence of the constants $C_1$ and $C_2$ upon $\alpha$ and consequently the dependence of $\tilde{K}$ and $\tilde{\varepsilon}$ upon $F_0$, if we desired. In this case, one would not use the direct maximum principle proofs of the pointwise derivative estimates (as we do below), but instead deduce them in the induction on scales (as Perelman does for the Ricci flow). 
\end{remark}
 
There are a few reasonable topologies one could use to say a given solution is $\varepsilon_0$-close to a model solution on a parabolic neighborhood. Based on the compactness result in the Appendix, we will use the following definition. Fix a small constant $\varepsilon_0 > 0$ and a large constant $K_0 < \infty$. Suppose $F : M \times [0, T) \to \mathbb{R}^N$ is a solution of the mean curvature flow and $(p_0, t_0)$ is a spacetime point. Set $Q_0 := H(p_0, t_0)$. Suppose that $F$ is defined in the intrinsic parabolic neighborhood $\hat P(p_0, t_0, K_0, K_0)$. Consider the rescaled solution 
\[
\tilde F(p, t) := Q_0F(p, t_0 + Q_0^{-2}(t - t_0)). 
\]
After rescaling, we have  $\tilde g(t_0) = Q_0^2 g(t_0)$, $\tilde H(p_0, t_0) = 1$, and the solution $\tilde F$ is defined in the intrinsic parabolic neighborhood $P(p_0,t_0, K_0, K_0)$.
\begin{definition}\label{varepsilon_close}
Let $F : M \times [0, T) \to \mathbb{R}^N$ be an $n$-dimension solution to mean curvature flow and $(p_0, t_0)$ a spacetime point satisfying $H(p_0, t_0) = 1$. Suppose $F$ is defined in the parabolic neighborhood $P(p_0, t_0, K_0, K_0)$ (i.e. $\partial M$, if it exists, satisfies $d_{g(t_0)}(p_0, \partial M) > K_0$ and $[t_0 - K_0, t_0] \subset [0, T)$). We will say the solution $F$ is $\varepsilon_0$-close in $P(p_0, t_0, K_0, K_0)$ to an ancient model solution if the following holds. We can find the following:
\begin{itemize}
\item an ancient model solution $\bar F: \bar M \times (-\infty, t_0] \to \mathbb{R}^N$;
\item a point $\bar p_0 \in \bar M$ with $H(\bar p_0, t_0) = 1$;
\item a smooth relatively compact domain $V$ such that $B_{\bar g( t_0)}(\bar p_0, K_0) \subset V \subset \bar M$;
\item a diffeomorphism $\Phi : V \to \Phi(V) \subset M$ such that $\Phi(\bar p_0) = p_0$ and $B_{g(t_0)}(p_0, K_0) \subset \Phi(V) \subset M$. 
\end{itemize}
Moreover, for each $t \in [t_0 - K_0, t_0]$, the immersions $F(\cdot, t) \circ \Phi$ and $\bar F(\cdot, t)$ are $\varepsilon_0$-close in $C^{[1/\varepsilon_0]}$ on $V$ with respect to the metric $\bar g := \bar g(t_0)$. Specifically, we think of $F(\cdot, t) \circ \Phi$ and $\bar F(\cdot, t)$ as $\mathbb{R}^N$-valued functions on $V$ and we require
\[
\sup_{t \in [t_0 -K_0, t_0]} \sup_{V}  \sum_{m =0}^{[1/\varepsilon_0]} \big|\bar \nabla^m\big(F(\cdot, t) \circ \Phi - \bar F(\cdot, t)\big) \big|_{\bar g}^2 < \varepsilon_0^2,
\]
where $\bar \nabla$ denotes the Levi-Civita connection of $\bar g$ on $\bar M$. 
\end{definition}
Of course, the definition applies to the rescaled solution if $H(p_0, t_0) \neq 1$. It follows from the evolution equation for $F$ that for any fixed pair of integers $(\ell, m)$, we can ensure that derivatives of the form $\frac{\partial}{\partial t}^\ell \bar \nabla^m (F(\cdot , t)\circ \Phi)$ are as close as we like to the corresponding derivatives of the model solution if we take $\varepsilon_0 \leq \frac{1}{2\ell + m}$ sufficiently small. 

In the final step of the proof of the canonical neighborhood theorem, we will need to make use of the following distance distortion estimate. The lemma is analogous to Lemma 8.3 in Perelman's work \cite{Per02a}. To the author's knowledge, this type of argument is originally due to Hamilton. By our pinching estimate, the product $Hh$ is comparable to $\mathrm{Ric}$ and this allows us follow Perelman's proof of the estimate for the Ricci flow. 

\begin{lemma}\label{dist_est}
Suppose $F : M \times (t_1, t_2) \to \mathbb{R}^{n+1}$ is a complete, convex, $n$-dimensional solution of mean curvature flow satisfying $|h|^2 \leq \beta^2 H^2$ for some $0 < \beta < 1$. Suppose we have a bound $H(\cdot, t) \leq \Lambda(t)$ for each $t \in (t_1, t_2)$. Then for any $p, q \in M$ and $t \in (t_1, t_2)$, we have
\[
0 \leq - \frac{d}{dt} d_{g(t)}(p, q) \leq C(n, \beta)\Lambda(t). 
\]
\end{lemma}
\begin{proof}
Recall that the evolution of the metric is given by $\frac{\partial}{\partial t} g_{ij} = -2 \langle \vec{H}, A_{ij} \rangle = -2 H h_{ij}$. Let us fix a time $t_0$ and restrict our attention to $(M, g(t_0))$. Fix some points $p, q \in M$. Set $\ell := d_{g(t_0)}(p, q)$ and let $\gamma : [0, \ell] \to M$ be a minimizing, unit-speed geodesic between $p$ and $q$. Set $X(s) = \gamma'(s)$. Then the derivative of the distance is given by
\[
\frac{d}{dt} d_{g(t)}(p, q) \Big|_{t = t_0}= - \int_0^\ell H h(X, X) \, ds. 
\]
We have assumed $h \geq 0$ and the inequality $|h|^2 \leq \beta^2 H^2$ implies $\lambda_n \leq \beta H$, where $\lambda_n$ denotes the maximum eigenvalue of $h$. From the identity above, together with $0 \leq Hh(X,X) \leq \beta H^2 < \Lambda(t_0)^2$, we first derive the crude estimate
\[
0 \leq - \frac{d}{dt} d_{g(t)}(p, q) \Big|_{t = t_0} \leq \Lambda(t_0)^2 \ell.  
\]
By the Gauss equation $Hh(X,X) = \mathrm{Ric}(X,X) + h^2(X,X)$. On the one hand, this gives $\mathrm{Ric}(X,X) \leq H h(X,X) < \Lambda(t_0)^2$. On the other hand, since $h^2(X, X) \leq \beta H h(X,X) $, it follows that $Hh(X,X) \leq \frac{1}{1-\beta} \mathrm{Ric}(X,X)$, and we have 
\[
0 \leq - \frac{d}{dt} d_{g(t)}(p, q) \Big|_{t = t_0} \leq  \frac{1}{1-\beta} \int_0^\ell \mathrm{Ric}(X,X) \, ds. 
\]
Having established the inequality above and the upper bound $\mathrm{Ric}(X,X) \leq \Lambda(t_0)^2$, we now proceed as Perelman does in \cite{Per02a}. Since $\gamma$ is a minimizing geodesic, for any vector field $V(s)$ defined along $\gamma$ and orthogonal to $X$, the second variation formula for the energy of $\gamma$ gives
\[
0 \leq \int_0^{\ell} |\nabla_XV|^2 - R(X, V, X, V) \, ds. 
\]
Consider a distance $r_0 \in (0, \frac{\ell}{2}]$.  Let $e_1(s), \dots, e_n(s)$ be a parallel orthonormal frame along $\gamma$ with $e_n(s) = X(s)$. Define a function 
\[
f(s) = \begin{cases} \frac{s}{r_0} & s \in [0, r_0] \\ 1 & s \in [r_0, \ell - r_0] \\ \frac{\ell -s}{r_0} & s \in [\ell - r_0, \ell] \end{cases}. 
\]
Plugging $V_i(s) = f(s)e_i(s)$ for $i = 1, \dots, n-1$ into the second variation formula and summing implies 
\begin{align*}
0  \leq \int_0^{r_0}\frac{n-1}{r_0^2} - \frac{s^2}{r_0^2}\mathrm{Ric}(X,X) \, ds - \int_{r_0}^{\ell - r_0} \mathrm{Ric}(X,X) \, ds  + \int_{\ell - r_0}^{\ell} \frac{n-1}{r_0^2} - \frac{(\ell - s)^2}{r_0^2} \mathrm{Ric}(X, X) \, ds. 
\end{align*}
Rearranging and using our upper bound for $\mathrm{Ric}$, we get
\begin{align*}
\int_0^{\ell} \mathrm{Ric}(X,X) \, ds &\leq 2(n-1)r_0^{-1}+   \int_0^{r_0} (1- \frac{s^2}{r_0^2})\mathrm{Ric}(X,X) \, ds+ \int_{\ell - r_0}^{\ell} (1- \frac{(\ell - s)^2}{r_0^2})\mathrm{Ric}(X, X) \, ds \\
& \leq 2(n-1)r_0^{-1} + 2\Lambda(t_0)^2 \int_0^{r_0} (1 - \frac{s^2}{r_0^2}) \, ds \\
&  = 2(n-1)r_0^{-1} + \frac{4}{3} \Lambda(t_0)^2r_0.
\end{align*}
Therefore, for any $0 < r_0 \leq \frac{\ell}{2}$, we have
\[
0 \leq - \frac{d}{dt} d_{g(t)}(p, q) \Big|_{t = t_0} \leq \frac{1}{1-\beta}\big( 2(n-1)r_0^{-1} + \frac{4}{3}n \Lambda(t_0)^2r_0\big). 
\]
If $\Lambda(t_0)^{-1} \leq \frac{\ell}{2}$, then taking $r_0 = \Lambda(t_0)^{-1}$ gives the desired conclusion. If $\Lambda(t_0)^{-1} > \frac{\ell}{2}$, then $\ell < 2 \Lambda(t_0)^{-1}$ and our crude estimate established earlier suffices. This completes the proof. 
\end{proof}

We will use the lemma above in conjunction with Hamilton's Harnack inequality for mean curvature flow. 

\begin{theorem}[R. Hamilton \cite{Ham95a}]
Assume $F : M \times (0, T) \to \mathbb{R}^{n+1}$ is a complete, weakly convex solution of mean curvature flow such that 
\[
\sup_{(p, t) \in M \times (\tau, T)} H(p, t) < \infty
\]
for all $\tau \in (0, T)$. Then, for all $(p, t) \in M \times (0, T)$ and any tangent vector $v \in T_pM$, 
\[
\frac{\partial}{\partial t} H  + \frac{1}{2t} H + 2\langle \nabla H, v \rangle + h(v, v)\geq 0.
\]
In particular, the quantity $\sqrt{t}H(p, t)$ is nondecreasing for each point $p \in M$ along the flow. 
\end{theorem}


\section{Proof of Theorem 1.1}\label{canon}

Our proof closely follows the proof of Theorem 7.2 in \cite{Bre19}, which is an adaptation of Perelman's work by Brendle. We will argue by contradiction and induction on scales, as introduced by Perelman. The theorem is stated only for $\varepsilon_0$ sufficiently small and $K_0$ sufficiently large depending upon the dimension and $F_0$. In Lemma \ref{step1} below, we determine a constant $\eta = \eta(F_0) < \infty$ for the pointwise derivative estimates. By the main result of \cite{Naf19b}, there exists a constant $\alpha = \alpha(\eta) = \alpha(F_0)>0$ so that any singularity model with these pointwise derivative estimates is $\alpha$-noncollapsed. We begin by fixing some $\varepsilon > 0$ small and letting $C_1 = C_1(n, \alpha(F_0), \varepsilon)$ and $C_2 = C_2(n, \alpha(F_0), \varepsilon)$ be the constants appearing in Proposition \ref{model_nbhds}. We will assume throughout the proof that $\varepsilon$ is sufficiently small depending upon certain universal constants that arise in the proof. We let $C(n)$ denote an arbitrary constant depending upon the dimension, which may change from line to line. We will prove the theorem assuming $\varepsilon_0$ is much smaller than $\varepsilon$ and $K_0 \geq 16C_1$. 

We take a moment to remind the reader that $H$ is shorthand for $|\vec H|$. Now two traces of the Gauss equation give $\mathrm{scal} = H^2 - |A|^2$. Together with the pinching assumption $\frac{1}{n} H^2 \leq |A|^2 \leq \tilde{c}_2 H^2$, this implies that $(1 - \tilde{c}_2) H^2 \leq \mathrm{scal}  \leq (1 - \frac{1}{n}) H^2$. Hence, in what follows, bounds for extrinsic curvature give bounds for the intrinsic curvature and vice versa. 

Now, if the assertion of the theorem is false, then there exists a sequence of spacetime points $(p_j, t_j)$ with the following properties:
\begin{enumerate}
\item[(i)] $Q_j := H(p_j, t_j) \geq j$. 
\item[(ii)] The solution is not $\varepsilon_0$-close in the parabolic neighborhood $\hat P(p_j, t_j, K_0, K_0)$ to any ancient model solution. 
\end{enumerate}
After point-picking process, we can further assume that the spacetime points $(p_j, t_j)$ have the additional property:  
\begin{enumerate}
\item[(iii)] If $ t \leq t_j$ and $(p, t)$ satisfies $H( p, t) \geq 4 Q_j$, then the solution is $\varepsilon_0$-close in the parabolic neighborhood $\hat P(p,  t, K_0, K_0)$ to an ancient model solution.  
\end{enumerate}
\noindent \textbf{Proof of point-picking.}
First, we ensure the flow is well defined on any $\hat{P}(p, t, K_0, K_0)$ if $H(p, t)$ is large enough. To see this, let $\bar{H}(s) := \sup_{M \times [0, s]} H$ and let $j_0 := \max\{\bar{H}(\frac{T}{2}), (\frac{K_0}{T})^\frac{1}{2}\}$. Of course $\bar{H}(s) < \infty$ for each $s \in [0, T)$ since $M$ is compact. Then if $(p, t)$ is any spacetime point with $H(p, t) \geq 2j_0$, we will have $\hat{P}(p, t, K_0, K_0) \subset M \times [0, T)$. Indeed, $j_0 \geq (\frac{K_0}{T})^{\frac{1}{2}}$ implies $\frac{K_0}{4j_0^2} \leq \frac{T}{4}$ and $H(p,t) \geq 2j_0 \geq 2 \bar{H}(\frac{T}{2})$ implies $t \geq \frac{T}{2}$. Therefore $t - K_0 H(p, t)^{-2} \geq \frac{T}{2} - \frac{K_0}{4j_0^2} \geq \frac{T}{4}$. Now restricting our attention to $j \geq j_0$, suppose the sequence of points $(p_j, t_j)$ do not yet satisfy property (iii). Then for each $j$, examine $M \times [0, t_j]$ to see if there exists a point $(p, t)$ such that $H(p, t) \geq 4 Q_j$, but the flow in $\hat{P}(p, t, K_0, K_0)$ is not $\varepsilon_0$-close to an ancient model solution. If such a point exists, denote it by $(p_{j, 1}, t_{j, 1})$. Otherwise, just let $(p_{j, 1}, t_{j,1}) = (p_j, t_j)$. The new sequence $(p_{j, 1}, t_{j, 1})$ satisfies property (i) (in the sense that $H(p_{j,1}, t_{j,1}) \geq j \to \infty$) and property (ii) by assumption. Now we repeat this procedure for the new sequence. We examine $M \times [0, t_{j, 1}]$ for a point violating property (iii) and obtain a new sequence $(p_{j,2}, t_{j,2})$ as before. For each $j$, this procedure produces a sequence of points $(p_{j, k}, t_{j, k})$ which must become constant for $k$ sufficiently large (depending upon $j$). Indeed, if after $k$ iterations, we have made $k$ replacements, then we must have $H(p_{j,k}, t_{j,k}) \geq 4^kQ_j \geq 4^k j$. On the other hand, since we have $t_{j,k} \leq t_j$, we also have $H(p_{j,k}, t_{j,k}) \leq \bar{H}(t_j)$. So this process can only continue as long as $4^k \leq \frac{1}{j} \bar{H}(t_j)$. Therefore, the desired sequence is a suitable diagonal subsequence of $(p_{j,k}, t_{j,k})$. \qed 

Under assumptions (i), (ii) and (iii), we will show that, after dilating the solution $F$ around the point $(p_j, t_j)$ by the factor $Q_j$, a subsequence of the rescaled solutions converges to an ancient model solution. Of course, this directly contradicts (ii) and thereby proves Theorem \ref{canonical-nbhd-thm}. The remainder of the proof, though perhaps well-understood by those familiar with \cite{Per02a, Bre19}, will take some pages. It may be helpful to note that in what follows, Lemma \ref{step1}, Lemma \ref{step2}, Proposition \ref{step3}, and Lemma \ref{step4} are essentially concerned with showing we can extract a complete limit at time zero. Then in Lemmas \ref{step5} and \ref{step6}, we show the limit can be extended backwards in time to a complete ancient solution, giving us the desired ancient model. 

In this first lemma, we begin with several a priori estimates for the flow. 

\begin{lemma}[a priori estimates \cite{Ngu18}]\label{step1}
There exists a large constant $\eta \geq 100$, depending upon the initial immersion $F_0$, such that $|\nabla A| \leq \eta H^2$, $|\nabla^2 A| \leq \eta H^3$, $|\nabla H| \leq \frac{\eta}{10} H^2$ and $|\frac{\partial}{\partial t} H|\leq \frac{\eta}{10} H^3$ hold on $M \times [0, T)$. Additionally, for every $\delta > 0$, there exists a constant $C_{\delta}$ (also depending upon $F_0$) such that $|A|^2 \leq \big(\frac{1}{n-1} + \delta) H^2 + C_{\delta}$ holds on $M \times [0, T)$. 
\end{lemma}

\begin{proof}  
We can find $a := a(F_0) > 0$ so that the estimate $|A|^2 + a \leq \tilde{c}_2 H^2$ holds initially. This inequality is preserved along the flow (by \cite{AB10}) and hence $H^2 \geq \frac{a}{\tilde{c}_2} > 0$ holds for all $t \in [0, T)$. Now the pointwise derivative estimates $|\nabla A| \leq \eta H^2$ and $|\nabla^2 A| \leq \eta H^3$ follow from the uniform lower bound for $H$ and Theorems 3.1 and 3.2 in \cite{Ngu18} for $\eta = \eta(F_0) \geq 100$ sufficiently large. 

Recall that $(\frac{\partial}{\partial t} \vec{H})^\perp = \Delta \vec{H} + \langle A_{pq}, \vec{H}\rangle A_{pq}$ (note the remark below). This implies $|(\frac{\partial}{\partial t} \vec{H})^\perp| \leq C(n)(|\nabla^2 A|+ |A|^2|\vec H|)$. Taking $\eta$ larger if necessary, this implies $|\nabla \vec{H}| \leq \frac{\eta}{100} H^2 $ and $|(\frac{\partial}{\partial t} \vec{H})| \leq  \frac{\eta}{100} H^3$ hold for all $t \in [0, T)$. Using Kato's inequality
implies $|\nabla H| \leq \frac{\eta}{10} H^2$ and $|\frac{\partial}{\partial t} H| \leq \frac{\eta}{10} H^3$ hold for $t \in [0, T)$, which gives the first set of estimates asserted in the lemma. 

Finally, the second assertion of the lemma is precisely the cylindrical estimate, Theorem 4.1, in \cite{Ngu18}.
\end{proof}

\begin{remark}
The evolution equations in higher codimension can be found in \cite{Naf19a}. In the above, and in what follows, we will write $\nabla$ for the natural covariant derivative acting on normal-bundle-valued tensors such as the second fundamental form $A$ and the mean curvature vector $\vec{H}$. In \cite{Naf19a}, we denoted this connection by $\nabla^\perp$, which is sometimes useful in order to distinguish it from the tangential component of the ambient connection of $\mathbb{R}^N$. In what follows however, the distinction is mostly irrelevant and we opt for simpler notation. See the Appendix for a slight elaboration on this point. 
\end{remark}

\begin{remark}
Although it is more convenient to use that the pointwise derivative estimates have direct proofs in our setting (and imply the cylindrical estimate), we could proceed differently. Although we will not pursue it here, an alternative approach would use that these estimates hold on the model solutions by \cite{HK17a} (for a uniform $\eta = \eta(\alpha)$). Then via the induction on scales, we may assume these estimates hold wherever the curvature is sufficiently large, which would suffice to prove the main theorem. This is the approach Perelman originally took for the Ricci flow (where, for instance, no direct proof of pointwise derivative estimates is yet known). 
\end{remark}

In the next lemma, we use the derivative estimates to establish short-range curvature estimates. We then obtain higher order pointwise derivative estimates for the second fundamental form.

\begin{lemma}[short-range curvature control]\label{step2}
As long as the flow is defined in the parabolic neighborhood $\hat{P}(\tilde{p}, \tilde{t}, \frac{1}{4\eta}, \frac{1}{4\eta})$, the scalar mean curvature satisfies $\frac{1}{4} H(\tilde{p}, \tilde{t}) \leq H \leq 4 H(\tilde{p}, \tilde{t})$ on $\hat{P}(\tilde{p}, \tilde{t}, \frac{1}{4\eta}, \frac{1}{4\eta})$. Consequently, for every $k$, there is a constant $C(k, n, \eta)$ such that $|\nabla^k A|(\tilde{p}, \tilde{t}) \leq C(k, n, \eta) H(\tilde{p}, \tilde{t})^{k+1}$. 
\end{lemma}

\begin{proof} Suppose $\hat{P}(\tilde{p}, \tilde{t}, \frac{1}{4\eta}, \frac{1}{4\eta})\subset M \times [0, T)$ (i.e. $\tilde{t} - \frac{1}{4\eta} H(\tilde{p}, \tilde{t})^{-2} \geq 0$). Assume $H(\tilde p, \tilde t \, ) = r_0^{-1}$. Suppose $(p, t)$ is a spacetime point satisfying $d_{g(\tilde t)}(p, \tilde p) \leq \frac{1}{4 \eta} r_0$ and $0 \leq \tilde t - t \leq \frac{1}{4\eta} r_0^2$. Integrating the spacial derivative estimate $|\nabla H| \leq \frac{\eta}{10} H^2$ gives 
\[
\Big|\frac{1}{H(\tilde p, \tilde t )}- \frac{1}{H(p, \tilde t )}  \Big| \leq \eta\, d_{g(\tilde t )}(p, \tilde p) \leq \frac{1}{4} r_0.
\]
Integrating the time derivative estimate $|\frac{\partial}{\partial t} H|\leq \frac{\eta}{10} H^3$ gives 
\[
\Big|\frac{1}{H(p, \tilde t )^2}- \frac{1}{H(p, t )^2}  \Big| \leq \eta \,|\tilde t - t| \leq \frac{1}{4} r_0^2.
\]
Combining these estimates with the assumption $H(\tilde p, \tilde t ) = r_0^{-1}$ shows that $\frac{1}{4} r_0^{-1} \leq H(p, t) \leq 4 r_0^{-1}$ for all $(p , t)$ in the parabolic neighborhood $P(\tilde p, \tilde t, \frac{r_0}{4\eta}, \frac{r_0^2}{4 \eta}) = \hat P(\tilde p, \tilde t, \frac{1}{4\eta}, \frac{1}{4 \eta})$. Now standard interior estimates for mean curvature flow imply that $|\nabla^k A|(\tilde p, \tilde t) \leq C(k, n, \eta) H(\tilde p, \tilde t)^{k+1}$.
\end{proof}

In the next proposition, we establish uniform bounds for $Q_j^{-1}H$ at bounded distance from $p_j$ at time $t_j$. This so-called long-range curvature estimate is by far the most involved step in the proof of the canonical neighborhood theorem. Establishing uniform bounds on $Q_j^{-1}H$ at bounded distance is important for extracting a complete limit at time zero. 

\begin{proposition}[long-range curvature control]\label{step3}
For all $\rho > 0$, 
\[
\mathbb{M}(\rho) := \limsup_{j\to \infty}  \sup_{p \in B_{g(t_j)}(p_j, Q_j^{-1}\rho )} Q_j^{-1} H(p, t_j)
\]
satisfies $\mathbb{M}(\rho) < \infty$. 
\end{proposition}

\begin{proof}
Note that $\mathbb{M}(\rho)$ is monotone increasing. By Lemma \ref{step2} (setting $(\tilde p, \tilde t) = (p_j, t_j)$), we have $\mathbb{M}(\rho) \leq 4$ for $0 < \rho \leq \frac{1}{4\eta}$. Define
\[
\rho_\ast := \sup \{ \rho \geq 0 : \mathbb{M}(\rho) < \infty \}. 
\]
Evidently, $\rho_\ast \geq \frac{1}{4\eta}$. Suppose, for sake of contradiction, that $\rho_\ast < \infty$. We will prove the proposition in five steps. Since the argument is lengthy, for readers unfamiliar with \cite{Per02a, Bre19}, here is an overview of the strategy.

\underline{Strategy of proof:}
\begin{enumerate}
\item By hypothesis, we have uniform bounds on $Q_j^{-1} H$ up to $g(t_j)$-distance $Q_j^{-1}\rho_\ast$ around $p_j$. After dilating by $Q_j$, we can extract a local limit immersion $F_{\infty}$ on a Riemannian geodesic ball $B_{g_{\infty}}(p_{\infty}, \rho_\ast)$ which is codimension one and convex. Because $\mathbb{M}(\rho_\ast) =\infty$, the limit contains a unit speed geodesic $\gamma_{\infty}(s)$, $s \in [0, \rho_\ast)$, (between $p_{\infty}$ and the boundary) along which the mean curvature $H_{\infty}(\gamma_{\infty}(s))$ blows up. 
\item We use the short-range curvature estimate to conclude $H_{\infty}(\gamma_\infty(s)) \gtrsim (\rho_\ast -s)^{-1}$ as $s \to \rho_\ast$. 
\item Using the induction on scales, we show that the end of the geodesic $\gamma_{\infty}([s_\ast, \rho_\ast))$ (for a suitably large $s_\ast \in [0, \rho_\ast)$) is contained in a domain $U \subset B_{g_{\infty}}(p_{\infty}, \rho_\ast)$ which is a union of (infinitely many) $C(n)\varepsilon$-necks.
\item It follows that through each point $\gamma_{\infty}(s)$, $s \in [s_\ast, \rho_\ast)$, there is a unique (nearly round) CMC sphere denoted $\Sigma_{u(s)} \subset U$ and the region $U$ is canonically foliated by such CMC spheres. We observe that $\mathrm{area}_{g_{\infty}}(\Sigma_{u(s)}) \sim H_{\infty}(\gamma_{\infty}(s))^{-n+1}$. Using that $B_{g_\infty}(p_\infty, \rho_\ast)$ has positive curvature, we show (Bishop-Gromov) area monotonicity of the CMC spheres implies that $H_{\infty}(\gamma_{\infty}(s)) \lesssim (\rho_\ast - s)^{-1}$ as $s \to \rho_\ast$.
\item Finally, by Steps 2 and 4, we have $H_{\infty}(\gamma_{\infty}(s)) \sim (\rho_\ast - s)^{-1}$ and $\mathrm{area}_{g_{\infty}}(\Sigma_{u(s)}) \sim (\rho_\ast - s)^{n-1}$. Roughly, this means after dilating the limit about $\gamma_{\infty}(s)$ as $s \to \rho_\ast$ by $(\rho_\ast - s)^{-1}$ and using rigidity in the area monotonicity, we can construct a local solution of mean curvature flow that has a final time slice which is a piece of a cone. This contradicts the strong maximum principle for the flow. 
\end{enumerate}
We now proceed with the proof. 

\underline{Step 1:} We begin by extracting a local geometric limit. By definition of $\rho_\ast$, for every $0 < \rho < \rho_\ast$, we have uniform bounds 
\[
\sup_{p \in B_{g(t_j)}(p_j, Q_j^{-1}\rho)}Q_j^{-1}H(p, t_j) \leq C(\rho).
\]
The higher order derivative estimates of Lemma \ref{step2} imply uniform bounds 
\[
\sup_{p \in B_{g(t_j)}(p_j, Q_j^{-1}\rho)} Q_j^{-k -1}|\nabla^kA|(p, t_j) \leq C(k, n, \eta, \rho)
\]
for each $0 < \rho < \rho_\ast$. For each $j$, define an immersion $\tilde F_j : B_{g(t_j)}(p_j, Q_j^{-1}\rho_\ast ) \to \mathbb{R}^N$ by 
\[
\tilde F_j(p) = Q_j\big(F(p, t_j) - F(p_j, t_j)\big). 
\]
Let $\tilde g_j := Q_j^2 g(t_j)$ denote the rescaled metric. Then $B_{g(t_j)}(p_j, Q_j^{-1}\rho_\ast ) = B_{\tilde g_j}(p_j, \rho_\ast)$. So $\tilde F_j : (M, \tilde g_j, p_j) \to (\mathbb{R}^N, g_{\mathrm{flat}}, 0)$ is a sequence of pointed isometric immersions and, for all $0 < \rho < \rho_\ast$ and for each $j$, we have uniform estimates for the second fundamental form of $\tilde F_j$ and its derivatives on $B_{\tilde g_j}(p_j, \rho)$. By Proposition \ref{immersion_compactness} in the Appendix, we can pass to a local limit. After passing to a subsequence, the sequence of immersions $\tilde F_j : (B_{\tilde g_j}(p_j, \rho_\ast), \tilde g_j, p_j) \to (\mathbb{R}^N, g_{\mathrm{flat}}, 0)$ converges on compact subsets to a locally defined pointed isometric immersion $F_{\infty} : (B_{\infty}, g_{\infty}, p_{\infty}) \to (\mathbb{R}^N, g_{\mathrm{flat}}, 0)$, where $B_{\infty} = B_{g_{\infty}}(p_{\infty}, \rho_\ast)$ is a geodesic ball in the metric $g_{\infty}$. In particular, $d_{g_{\infty}}(p_{\infty}, \partial B_{\infty}) \geq \rho_\ast$. 

The limit must satisfy $H_{\infty} > 0$. Indeed, since the (scale-invariant) derivative estimate $|\nabla H_{\infty}| \leq \eta H_{\infty}^2$ passes to the limit and $H_{\infty}(p_{\infty}) = 1$, this follows from integration. Now because the limit satisfies $H_{\infty} > 0$, the planarity estimate in \cite{Naf19a} (in particular, see Proposition 2.5 in \cite{Naf19a}) implies $F_{\infty}(B_\infty)$ is contained in an affine ($n+1$)-dimensional subspace. So we may write $F_{\infty} : B_\infty \to \mathbb{R}^{n+1}$. Since the limit is codimension one, the cylindrical estimate, $|A_{\infty}|^2 \leq \frac{1}{n-1} H_{\infty}^2$, implies the limit is weakly convex. 
 
By assumption $\mathbb{M}(\rho_\ast) = \infty$. Thus, we can find a sequence of points $q_j \in B_{g(t_j)}(p_j, Q_j^{-1}\rho_\ast)$ such that 
 \[
 Q_j^{-1} H(q_j, t_j) \to \infty \quad \text{ and } \quad \rho_j := Q_j d_{g(t_j)}(p_j, q_j) \to \rho_\ast.
 \] 
 For each $j$, let $\gamma_j : [0,  Q_j^{-1} \rho_j] \to B_{g(t_j)}(p_j, Q_j^{-1}\rho_\ast )$ be a unit-speed length-minimizing $g(t_j)$-geodesic between the points $p_j$ and $q_j$. The geodesics $s \mapsto \gamma_j(Q_j^{-1} s)$ for $s \in [0, \rho_j]$ converge locally to a unit-speed geodesic $\gamma_{\infty} : [0, \rho_\ast) \to B_\infty$ missing its terminal point.  
 
 \underline{Step 2:} The pointwise derivative estimate, $|\nabla H_{\infty}| \leq \eta H_{\infty}^2$, implies 
 \[
 H_{\infty}(\gamma_{\infty}(s)) = \lim_{j \to \infty} Q_j^{-1} H(\gamma_j(Q_j^{-1}s), t_j) \geq (2\eta(\rho_\ast - s))^{-1} \geq 16
 \]
 for $s \in [\rho_\ast - \frac{1}{32 \eta}, \rho_\ast)$. Indeed, suppose for some $\tilde s \in  [\rho_\ast - \frac{1}{32 \eta}, \rho_\ast)$, we have $H_{\infty}(\gamma_{\infty}(\tilde s)) < (2\eta(\rho_\ast - \tilde s))^{-1}$. Integrating the derivative estimate as in Lemma \ref{step2}, we find that 
 \[
 H_{\infty}(\gamma_{\infty}(s)) \leq \frac{H_{\infty}(\gamma_{\infty}(\tilde s))}{1 - \eta (s -  \tilde{s}) H_{\infty}(\gamma_\infty(\tilde s))} \leq (\eta(\rho_\ast - \tilde s))^{-1} < \infty
 \]
for $s \in (\tilde s, \rho_\ast)$. In particular, $\lim_{s \to \rho_\ast} H_{\infty}(\gamma_{\infty}(s)) < \infty$, which contradicts $\lim_{j \to \infty} Q_j^{-1} H(q_j, t_j) \to \infty$.

\underline{Step 3:} Next, we show that for all $s$ sufficiently close to $\rho_\ast$, the point $\gamma_{\infty}(s)$ lies at the center of a neck. Let 
\[
s_\ast := \max\Big\{ \frac{32C_1 \eta}{32C_1 \eta + 1} \rho_\ast\;,\; \rho_\ast - \frac{1}{32\eta}\Big\}. 
\]
Consider $s \in [s_\ast, \rho_\ast)$ so that $32C_1\eta(\rho_\ast - s) \leq s$.  Consider the point $x_j:= \gamma_j(Q_j^{-1}s)$. By Step 2, $H(x_j, t_j) \geq 4 Q_j$ if $j$ is sufficiently large. By property (iii), we can find a neighborhood of $x_j$ containing $B_{g(t_j)}(x_j, \frac{1}{2}K_0 H(x_j, t_j)^{-1})$ on which the immersion $F(\cdot , t_j)$ is $\varepsilon_0$-close to a time slice of an embedding of a model solution. Under the assumptions $K_0 \geq 16C_1$ and $\varepsilon_0 \ll \varepsilon$ (by a factor depending on the dimension) Definition \ref{necks_caps}, Definition \ref{varepsilon_close}, and Proposition \ref{model_nbhds} imply the point $x_j$ has a canonical neighborhood $U_j$ with the following properties:
\begin{enumerate}
\item[$\bullet$] $U_j$ is either a $2\varepsilon$-neck, a $2\varepsilon$-cap, or a closed manifold diffeomorphic to $S^n$.
\item[$\bullet$] $B_{g(t_j)}(x_j, (2C_1)^{-1} H(x_j, t_j)^{-1}) \subset U_j \subset B_{g(t_j)}(x_j, 2C_1H(x_j, t_j)^{-1})$.
\item[$\bullet$] In $U_j$, the mean curvature satisfies $(2C_2)^{-1} H(x_j, t_j) \leq H \leq 2C_2H(x_j, t_j)$.
\end{enumerate} 
Let us show $U_j$ must be a $2\varepsilon$-neck. Since $\mathbb M(s) < \infty$, the ratio $H(q_j, t_j)/H(x_j, t_j) \to \infty$ as $j \to \infty$. Therefore, $H(q_j, t_j) \geq 4C_2H(x_j, t_j)$ if $j$ is sufficiently large and it follows from the third item above that $q_j \not \in U_j$. The estimate in Step 2 gives $8C_1 H(\gamma_{\infty}(s))^{-1} \leq 16C_1\eta (\rho_\ast -  s) < s$, by definition of $s_\ast$. Hence $4C_1 H(x_j, t_j)^{-1} \leq Q_j^{-1} s  = d_{g(t_j)}(p_j, x_j)$ if $j$ is sufficiently large and it follows from the second item above that $p_j \not \in U_j$. These considerations evidently imply $U_j$ is not a closed manifold diffeomorphic to $S^n$. If $U_j$ is a $2\varepsilon$-cap, then the geodesic $\gamma_j$  enters and exits the cap. However, by definition, the boundary of a $2\varepsilon$-cap is a central cross-sectional sphere of a $2\varepsilon$-neck. If $\varepsilon$ is sufficiently small depending upon the dimension, this contradicts the fact that $\gamma_j$ is minimizing. 

In summary, for each $s \in [s_{\ast}, \rho_\ast)$, the point $\gamma_j(Q_j^{-1}s)$ has a canonical neighborhood which is a $2\varepsilon$-neck for $j$ sufficiently large (depending upon $s$). Moreover, on each $2\varepsilon$-neck, we have the improved gradient estimate $|\nabla H| \leq C(n)\varepsilon H^2$. Passing to the limit for each such $s \in [s_{\ast}, \rho_\ast)$, we conclude the point $\gamma_{\infty}(s)$ lies at the center of a $C(n)\varepsilon$-neck in $(B_{\infty}, g_{\infty})$. As in Step 2, the improved gradient estimate gives
\[
H_{\infty}(\gamma_{\infty}(s)) \geq (C(n)\varepsilon (\rho_\ast - s))^{-1}
\]
for $s \in [s_\ast, \rho_\ast)$. 

\underline{Step 4:} In this step and the next step, we follow the argument of Brendle in \cite{Bre18a}. By Step 3, every point $\gamma_{\infty}(s)$, for $s \in [s_\ast, \rho_\ast)$, lies at the center of an $C(n)\varepsilon$-neck. In the present context, a $C(n)\varepsilon$-neck is an extrinsic notion given by Definition \ref{necks_caps}, but the definition implies each point lies at the center of an intrinsic $C(n)\varepsilon$-neck in the sense used by Perelman in \cite{Per02b}. Let $U \subset B_{\infty}$ denote the connected region obtained by taking the union over all of the $C(n)\varepsilon$-necks centered at the points $\gamma_{\infty}(s)$ for $s \in [s_{\ast}, \rho_\ast)$.

The work of Hamilton \cite{Ham97} shows that each of these $C(n)\varepsilon$-necks admits a canonical foliation by constant mean curvature spheres (at least away from the boundary of the neck). Moreover, the foliations of overlapping necks must agree and can be joined together. Consequently, the domain $U$ admits a foliation by a one-parameter family of CMC spheres, which we denote by $\Sigma_u$. We can arrange the parameter $u$ so that the CMC spheres are defined for $u \in (0, u_\ast]$ and so that as $u \to 0$, the spheres $\Sigma_u$ move away from the point $p_{\infty}$ and towards the end of the geodesic. We let $v : \Sigma_u \to \mathbb{R}$ denote the lapse function of the foliation. We can parametrize the foliation $P : S^{n-1} \times (0, u_\ast] \to U$ so that $v = |\frac{\partial P}{\partial u}|_{g_{\infty}}$. We can also express the leaves of the foliation as the level sets of the projection $\pi : U \to (0, u_{\ast}]$. In this case, $v = |\nabla \pi |_{g_{\infty}}^{-1}$.  After reparametrization, we may assume the average of $v$ over $\Sigma_u$ is $1$. Note that $\sup_{\Sigma_u}|v -1| \leq C(n)\varepsilon$. Finally, let $\nu := -v^{-1} \frac{\partial P}{\partial u} = -v \nabla \pi$ denote the unit normal to the foliation.

For each $s \in [s_\ast, \rho_\ast)$, the point $\gamma_{\infty}(s)$ contained in a CMC sphere $\Sigma_{u(s)}$ where $u(s) := \pi(\gamma_{\infty}(s))$. We may assume $u(s_\ast) = u_{\ast}$. Since $\gamma_{\infty}$ is the limit of a sequence of length minimizing geodesics $\gamma_j$ on $2\varepsilon$-necks, we must have $1 - C(n)\varepsilon\leq g_{\infty}(\nu, \gamma_{\infty}') \leq 1$. Since
\[
\frac{du}{ds}= g_{\infty}(\nabla \pi, \gamma_{\infty}') = - v^{-1} g_{\infty}(\nu, \gamma_{\infty}'), 
\]
this gives $ |u'(s) + 1| \leq C(n)\varepsilon$ and $1 - C(n)\varepsilon \leq u(s)(\rho_\ast - s)^{-1}\leq 1 + C(n)\varepsilon$ for $s \in [s_\ast, \rho_\ast)$. Because each CMC sphere $\Sigma_{u}$ must be close to a round sphere on a $C(n) \varepsilon$-neck, we have the estimates
\[
\sup_{q \in \Sigma_{u}} \mathrm{scal}_{g_\infty}(q) \leq (1 + C(n)\varepsilon) \inf_{q \in \Sigma_{u}} \mathrm{scal}_{g_\infty}(q) 
\]
and 
\[
\frac{1}{C(n)} \big(\sup_{q \in \Sigma_u} \mathrm{scal}_{g_{\infty}}(q)\big)^{-\frac{n-1}{2}} \leq \mathrm{area}_{g_{\infty}}(\Sigma_u) \leq C(n) \big(\inf_{q \in \Sigma_u} \mathrm{scal}_{g_{\infty}}(q)\big)^{-\frac{n-1}{2}} 
\]
for $u \in (0, u_\ast]$.

Recall that the cylindrical estimate $|A_{\infty}|^2 \leq \frac{1}{n-1}H_{\infty}^2$ implies $H_{\infty}^2$ and $\mathrm{scal}_{g_{\infty}}$ are comparable. In particular, from Step 3 we obtain 
\[
(\rho_\ast - s)^2 \mathrm{scal}_{g_{\infty}}(\gamma_{\infty}(s)) \geq (C(n) \varepsilon)^{-2} > 0 
\]
for $s \in [s_\ast, \rho_\ast)$. Since $u(s)$ and $\rho_\ast - s$ are comparable, we get 
\[
u^2\inf_{q \in \Sigma_u} \mathrm{scal}_{g_{\infty}}(q) \geq (C(n) \varepsilon)^{-2} > 0
\] 
for $u \in (0, u_\ast]$. 
This implies 
\[
\frac{\mathrm{area}_{g_{\infty}}(\Sigma_u)}{u^{n-1}} \leq C(n) \varepsilon^{n-1} < \infty
\]
for $u \in (0, u_\ast]$.  In particular, this implies that $\mathrm{area}_{g_{\infty}}(\Sigma_u) \to 0$ as $u \to 0$. 

We will now consider the extrinsic geometry of each of the leaves $\Sigma_u$ as hypersurfaces within $(B_{\infty}, g_{\infty})$. Let $\mathcal H$ and $\mathcal A$ denote the scalar mean curvature and second fundamental form of $\Sigma_u$ with respect to $g_{\infty}$ (not to be mistaken for $H_{\infty}$ and $A_{\infty}$). For each $u$, the mean curvature $\mathcal H = \mathcal H(u)$ of $\Sigma_u$ is constant. As we are on a $C(n)\varepsilon$-neck, the ratio $\big(\sup_{q \in \Sigma_u} \mathrm{scal}_{g_{\infty}}(q)\big)^{-\frac{1}{2}}\mathcal H(u)$ is close to zero. 

Using the first variation formula for the mean curvature and the fact that $(B_{\infty}, g_{\infty})$ has nonnegative Ricci curvature, we obtain
\[
-\mathcal H'(u) = \Delta_{\Sigma_u} v + (|\mathcal A|^2 + \mathrm{Ric}_{g_{\infty}}(\nu, \nu))v \geq \Delta_{\Sigma_u} v  + \frac{1}{n-1} \mathcal H(u)^2 v.
\]
Taking the average over $\Sigma_u$ on both sides gives
\[
-\mathcal H'(u) \geq \frac{1}{n-1} \mathcal H(u)^2. 
\]
Now either $\mathcal H(u) \leq 0$ for all $u \in (0, u_\ast]$, or $\mathcal{H}(u_0)$ is positive for some $u_0 \in (0, u_\ast]$ (hence positive on $(0, u_0]$), and the differential inequality implies $(\frac{1}{\mathcal{H}(u)})' \geq \frac{1}{n-1}$ for $u \in (0, u_0]$. Integrating implies 
\[
0 < \mathcal H(u) \leq \frac{n-1}{u}
\]
for $u \in (0, u_0]$. In fact, we cannot have $\mathcal{H}(u) \leq 0$ for all $u$ since
\[
\frac{d}{du} \mathrm{area}_{g_{\infty}}(\Sigma_u) = \mathcal H(u) \int_{\Sigma_u} v  = \mathcal H(u)\,  \mathrm{area}_{g_{\infty}}(\Sigma_u), 
\]
and we know $\mathrm{area}_{g_{\infty}}(\Sigma_u) \to 0$ as $u \to 0$. In any case, we have
\[
\frac{d}{du}\big(u^{1-n} \mathrm{area}_{g_{\infty}}(\Sigma_u)\big) = u^{1-n} \mathrm{area}_{g_{\infty}}(\Sigma_u)\big(\mathcal H(u) - \frac{n-1}{u}\big) \leq 0. 
\]
Now because the function $u \mapsto u^{1-n} \mathrm{area}_{g_{\infty}}(\Sigma_u)$ is monotone decreasing, clearly 
\[
\liminf_{u \to 0} \frac{\mathrm{area}_{g_{\infty}}(\Sigma_u)}{u^{n-1}} > 0.
\]
This implies 
\[
\limsup_{u \to 0} \big(u^2\sup_{q \in \Sigma_u} \mathrm{scal}_{g_{\infty}}(q)\big)  < \infty. 
\]
To summarize, $u(s)^2 \mathrm{scal}_{g_{\infty}}(\gamma_{\infty}(s))$, and hence $u(s) H_{\infty}(\gamma_{\infty}(s))$, are bounded above and below as $s \to \rho_\ast$. Moreover, the ratio $u^{1-n}\mathrm{area}_{g_{\infty}}(\Sigma_u)$ is bounded and monotone, and therefore 
\[
\lim_{u \to 0} \frac{\mathrm{area}_{g_{\infty}}(\Sigma_u)}{u^{n-1}} = \kappa 
\]
for some positive constant $\kappa \in (0, C(n)\varepsilon^{n-1})$. In the next step, we will see that this essentially says the singularity at $\gamma_{\infty}(\rho_\ast)$ is modeled on a cone. For later use, let us choose a positive constant $L$ such that
\[
u \sup_{q \in \Sigma_u} H_{\infty}(q)  \leq L
\]
for $u \in (0, u_\ast]$. 

\underline{Step 5:} Now we can find a suitable sequence of rescalings of our original flow which converge to a local flow for which the final time slice is a piece of a metric cone. This will contradict the strong maximum principle. Here are the details. Choose a sequence of distances $s_{\ell} \in [s_\ast, \rho_\ast)$ with $s_{\ell} \to \rho_\ast$. Let 
\[
u_{\ell} := u(s_\ell),\qquad \hat{x}_{\ell} := \gamma_{\infty}(s_\ell) \in \Sigma_{u_\ell}.  
\]
Then $u_\ell \to 0$.  After passing to a subsequence, we can assume 
\[
u_\ell^{-1}(\rho_\ast - s_\ell)\to \hat \rho
\]
where $\hat{\rho}$ is a constant that is close to $1$. For $j$ sufficiently large, let 
\[
x_{\ell, j} := \gamma_j(Q_j^{-1} s_\ell) \in B_{g(t_j)}(p_j, Q_j^{-1} \rho_\ast)
\] 
so that $x_{\ell, j} \to \hat x_\ell$ as $j \to \infty$. Define 
\[
R_{\ell, j} := u_\ell Q_{j}^{-1}.
\]

Recall $\tilde g_j = Q_j^2 g(t_j)$ and the metrics $\tilde g_j$ converge to $g_{\infty}$. Now consider the rescaled metrics $\hat g_{\ell, j} := R_{\ell, j}^{-2} g(t_j) = u_\ell^{-2} \tilde g_{j}$ on the metric balls 
\[
\hat B_{\ell,j} := B_{\hat g_{\ell, j}}\Big(x_{\ell, j}, \frac{\rho_\ast - s_\ell}{2 u_\ell}\Big) = B_{\tilde g_{j}}\Big(x_{\ell,j}, \frac{\rho_\ast - s_\ell}{2}\Big) \subset B_{\tilde g_{j}}(p_{j}, \rho_\ast). 
\]
By our work in Step 1, $(\hat B_{\ell,j}, \hat g_{\ell, j})$ converges locally smoothly to $\big(B_{g_\infty}\big(\hat x_\ell, \frac{\rho_\ast - s_\ell}{2}\big), u_\ell^{-2} g_{\infty}\big)$ for each fixed $\ell$ as we let $j \to \infty$.  By our work in Step 3, if $j$ is large depending upon $\ell$, then every point $x \in \hat B_{\ell, j}$ lies at the center of a $C(n)\varepsilon$-neck and therefore on a canonical CMC sphere. So $\hat B_{\ell, j}$ is contained in a tube that is canonically foliated by CMC spheres. Because the foliation by CMC spheres is uniquely determined by the metrics (via the inverse function theorem) and because the metrics $\tilde g_j = Q_j^2 g(t_j)$ converge locally smoothly to $g_{\infty}$, each CMC sphere contained in $\hat B_{\ell, j}$ converges to a unique CMC sphere in $B_{g_\infty}\big(\hat x_\ell, \frac{\rho_\ast - s_\ell}{2}\big)$ as we let $j \to \infty$. It follows from comparability of $u(s)$ and $\rho_\ast - s$ that $B_{g_\infty}\big(\hat x_\ell, \frac{\rho_\ast - s_\ell}{2}\big) \subset \bigcup_{u \in (\frac{1}{3} u_\ell, \frac{5}{3} u_\ell)} \Sigma_u$ and conversely, if $ \frac{2}{3}u_\ell < u < \frac{4}{3}u_\ell$, then $\Sigma_u \subset B_{g_\infty}\big(\hat x_\ell, \frac{\rho_\ast - s_\ell}{2}\big)$. 

Let $\hat \Sigma^{(\ell, j)}_{\hat u}$ denote the one-parameter family of CMC spheres that foliate $\hat{B}_{\ell, j}$. After a translation, choice of sign, and suitable reparametrization (so that the lapse function $\hat v^{(\ell, j)}$ has average 1 on each CMC sphere) the foliation is defined at least for $\hat u \in (\frac{1}{3}, \frac{5}{3})$ with $\hat \Sigma^{(\ell, j)}_{\hat u} \subset \hat B_{\ell, j}$ if $\hat u \in (\frac{2}{3}, \frac{4}{3})$ and $\hat B_{\ell, j} \subset \bigcup_{\hat u \in (\frac{1}{3}, \frac{5}{3})} \Sigma^{(\ell, j)}_{\hat u}$. In particular, these choices determine the parameter $\hat u$ and consequently $\hat \Sigma^{(\ell, j)}_{\hat u}$ converges to $\Sigma_{u_{\ell}\hat{u}}$ as $j \to \infty$ for $\hat u \in (\frac{1}{3}, \frac{5}{3})$. Therefore
\begin{align*}
\lim_{j \to \infty} \frac{\mathrm{area}_{\hat g_{\ell, j}}(\hat \Sigma^{(\ell, j)}_u)}{\hat u^{n-1}} &= \frac{\mathrm{area}_{u_\ell^{-2} g_{\infty}}(\Sigma_{u_\ell \hat u})}{\hat u^{n-1}} = \frac{\mathrm{area}_{g_{\infty}}(\Sigma_{u_\ell \hat u })}{(u_\ell \hat u)^{n-1}},\\
\lim_{\ell \to \infty} \frac{\mathrm{area}_{g_{\infty}}(\Sigma_{u_\ell \hat u })}{(u_\ell \hat u)^{n-1}} &= \kappa. 
\end{align*}

Set $\tau := - \frac{1}{288 \eta L^2}$. Finally, consider the sequence of flows $\hat F_{\ell, j} : \hat B_{\ell, j} \times (\tau, 0] \to \mathbb{R}^N$ defined by
\[
\hat F_{\ell,j}(p, t) = R_{\ell, j}^{-1} \big(F(p, t_j + R_{\ell, j}^2 t) - F(x_\ell, t_j)\big). 
\]
For each $p \in \hat B_{\ell, j}$, we can find $\hat u \in (\frac{1}{3}, \frac{5}{3})$ with $p \in \Sigma^{(\ell, j)}_{\hat u}$. Assuming that $j$ is large enough depending upon $\ell$, it follows from our work in Step 4 that 
\[
Q_j^{-1}H(p, t_j) \leq  2 L u_\ell^{-1} \hat u^{-1} \leq 6 L u_\ell^{-1}. 
\]
For $t \in (\tau, 0]$, this implies that
\[
R_{\ell, j}^2 (-t) < R_{\ell, j}^2 (-\tau) = \frac{1}{288 \eta}  u_\ell^2 Q_j^{-2} L^{-2}  \leq \frac{1}{8 \eta H(p, t_j)^2}. 
\]
So we may use the short-range curvature estimates in Step 2 to obtain
\[
R_{\ell, j} H(p, t_j+ R_{\ell,j}^2 t) \leq  4 R_{\ell,j} H(p, t_j) \leq 24 L
\]
for $(p, t) \in \hat B_{\ell, j} \times (\tau, 0]$. The estimate above implies we have uniform bounds for the second fundamental forms of the flows $\hat F_{\ell, j}$ on the domains $\hat{B}_{\ell, j} \times (\tau, 0]$ assuming $j$ is sufficiently large depending upon $\ell$. For a suitable diagonal subsequence, we can now apply Corollary \ref{flow_compactness} to obtain a locally defined solution of mean curvature flow $\hat F_{\infty}$ in $\mathbb{R}^{n+1}$, which is defined on a parabolic neighborhood $\hat B_{\infty} \times (\tau, 0]$, where $\hat B_\infty = B_{\hat g_{\infty}}(\hat x_{\infty}, \frac{1}{2} \hat \rho)$. If along the subsequence $j$ is taken sufficiently large for each $\ell$, then the CMC foliations of $\hat B_{\ell, j}$ converge to a CMC foliation $\hat \Sigma_{\hat u}$ of suitable domain $\hat U$ in $\hat B_{\infty}$ with the property that 
\[
\frac{\mathrm{area}_{\hat g_{\infty}}(\hat \Sigma_{\hat u})}{\hat u^{n-1}} = \kappa
\]
independent of $\hat{u}$.

Let $\hat v$ and $\hat \nu$ denote the lapse function and inward-pointing unit normal of the foliation $\hat \Sigma_{\hat u}$. If $\hat{ \mathcal H}(\hat u)$ and $\hat{\mathcal{A}}$ denote the mean curvature and second fundamental form of $\hat \Sigma_{\hat u}$ with respect to $\hat g_{\infty}$, then the identities in Step 4 imply $\hat{\mathcal H}(\hat u) = \frac{n-1}{\hat u}$, $-\hat{\mathcal H}'(\hat u) =  \frac{1}{n-1} \hat{\mathcal H}(\hat u)^2$, $|\hat {\mathcal{A}}|^2 = \frac{1}{n-1} \hat{\mathcal H}(\hat u)^2$, and $\mathrm{Ric}_{\hat g_{\infty}}(\hat \nu, \hat \nu) = 0$. In particular, $\Delta_{\hat{\Sigma}_{\hat{u}}} (\hat{v} -1)+ \frac{n-1}{\hat{u}^2}(\hat{v}-1) =0$. The sphere $\hat{\Sigma}_{\hat{u}}$ is $C(n)\varepsilon$-close to a round sphere of radius $\lesssim \varepsilon \hat{u}$. This implies the operator $\Delta_{\hat{\Sigma}_{\hat{u}}} + \frac{n-1}{\hat{u}^2}$ is a small perturbation of $\Delta_{S^{n-1}}$, after dilating by $\varepsilon^{-1}\hat{u}^{-1}$. Because the average of $\hat{v} - 1$ is zero, this in turn implies $\hat v \equiv 1$, and hence $\hat U$ must be a piece of metric cone. Note, the estimate $\hat u^{1-n}\mathrm{area}_{\hat g_{\infty}}(\hat \Sigma_{\hat u}) \leq C(n)\varepsilon^{n-1}$ implies opening angle of the cone must be very small.  

Since the solution $\hat F_{\infty}$ is codimension 1, weakly convex, and we have $\mathrm{Ric}_{\hat g_{\infty}}(\hat \nu, \hat \nu) = 0$, the Gauss equation implies that the first eigenvalue of the second fundamental form of $F_{\infty}$ must vanish at time $t = 0$. By the strong maximum principle, this implies the solution locally splits a line (in fact, it implies $\hat{\nu}$ is parallel on $\hat{U}$). In particular, the metric induced by the immersion is locally a product, which is not compatible with the previous conclusion that geometry at time zero is conical (e.g. we have seen that $\mathrm{area}_{\hat g_{\infty}}(\hat \Sigma_{\hat u}) = \kappa \hat u^{n-1}$, but on a product $\mathrm{area}_{\hat g_{\infty}}(\hat \Sigma_{\hat u})$ is constant). This gives us the desired contradiction (we also refer the reader to Appendix A in \cite{HK17a}). We conclude 
\[
\mathbb M(\rho) :=  \limsup_{j\to \infty}  \sup_{p \in B_{g(t_j)}(p_j, Q_j^{-1}\rho )} Q_j^{-1} H(p, t_j) < \infty
\]
for each $\rho > 0$. This completes the proof of the long-range curvature estimate. 
\end{proof}

We can now extract a complete limit and show it must have bounded curvature. Let $\tilde F_j : M \times [-Q_j^2t_j, 0] \to \mathbb{R}^N$ be the rescaled flows given by 
\[
\tilde F_j(p, t) := Q_j\big(F(p, t_j + Q_j^{-2}t) - F(p_j, t_j)\big). 
\]
Let $\tilde{g}_j= Q_j^2 g(t_j)$ be the corresponding metrics at time zero.  

\begin{lemma}[extracting a complete limit at time zero]\label{step4}
After passing to a subsequence: 
\begin{itemize}
\item The sequence $(M, \tilde g_j, p_j)$ converges smoothly in the pointed Cheeger-Gromov sense to a complete Riemannian manifold $(M_{\infty}, g_{\infty}, p_{\infty})$.
\item The sequence of immersions $\tilde F_j(\cdot, 0)$ converges smoothly to a pointed isometric immersion $F_{\infty} : (M_{\infty}, g_{\infty}, p_{\infty}) \to (\mathbb{R}^N, g_{\mathrm{flat}}, 0)$, which is codimension one and satisfies the estimates $|\nabla A_{\infty}| \leq \eta H_{\infty}^2$, $|\nabla^2 A_{\infty}| \leq \eta H_{\infty}^3$, and $|A_{\infty}|^2 \leq \frac{1}{n-1}H_{\infty}^2$. 
\item Either $M_{\infty}$ is a closed manifold diffeomorphic to $S^n$ or every point $p \in M_{\infty}$ where $H_{\infty}(p) \geq 8$ has a neighborhood that is either a $2\varepsilon$-neck or a $2\varepsilon$-cap. Consequently $(M_{\infty}, g_{\infty})$ has nonnegative and bounded curvature. 
\end{itemize}
\end{lemma}

\begin{proof}
By Proposition \ref{step3}, for every $\rho \in (0, \infty)$ there exists a constant $C(\rho)$ such that $Q_j^{-1}H(p, t_j) \leq C(\rho)$  if $d_{g(t_j)}(p, p_j) < \rho Q_j^{-1}$ (for all $j$ large).  Hence by Lemma \ref{step2}, for every $\rho \in (0, \infty)$ and every integer $k = 0, 1, \dots$, there exists a positive constant $C(n, k, \eta, \rho)$ such that we have $Q_j^{-k-1} |\nabla^k A|(p, t_j) \leq C(n, k, \eta, \rho)$ if $d_{g(t_j)}(p, p_j) < \rho Q_j^{-1}$ (for all $j$ large). In other words, for each of the rescaled pointed isometric immersions $\tilde F_j(\cdot, 0)$ we have uniform estimates for the second fundamental form and each of its derivatives at bounded distance (for all $j$ large). After passing to a subsequence, Proposition \ref{immersion_compactness} implies convergence to $F_{\infty} : (M_{\infty}, g_{\infty}, p_{\infty}) \to (\mathbb{R}^N, g_{\mathrm{flat}}, 0)$. 

By Lemma \ref{step1}, in the limit we have $|\nabla A_{\infty}| \leq \eta H_{\infty}^2$, $|\nabla^2 A_{\infty}| \leq \eta H_{\infty}^3$, and $|A_{\infty}|^2 \leq \frac{1}{n-1}H_{\infty}^2$. The planarity estimate \cite{Naf19a} implies that $F_{\infty}(M_{\infty})$ is contained in an $(n+1)$-dimensional affine subspace of $\mathbb{R}^N$. In codimension one, the estimate $|A_{\infty}|^2 \leq \frac{1}{n-1}H_{\infty}^2$ implies the limit has nonnegative sectional curvature.

By condition (iii) and Proposition \ref{model_nbhds}, either $M_{\infty}$ is a closed manifold diffeomorphic to $S^n$ or every point $p \in M_{\infty}$ where $H_{\infty}(p) \geq 8$ has a canonical neighborhood which is either a $2\varepsilon$-neck or a $2\varepsilon$-cap. Note that if a point $p \in M_{\infty}$ lies on a $2\varepsilon$-cap, then the estimates in Proposition \ref{model_nbhds} imply the cap is incident to a $2\varepsilon$-neck of radius bounded by $C(n)C_2 H_{\infty}(p)^{-1}$. Thus, if the curvature of $(M_{\infty}, g_{\infty})$ is unbounded, then the limit must contain a sequence of $2\varepsilon$-necks of radii tending to zero. However, by an argument of Perelman, this is impossible in a complete Riemannian manifold with nonnegative sectional curvature. See Proposition 2.2 in \cite{ChZ06} for a detailed version of the argument. Consequently, $(M_{\infty}, g_{\infty})$ has bounded curvature. 
\end{proof}

\begin{lemma}[extending the limit backwards in time]\label{step5}
There exists $\tau^\ast \in [-\infty, 0)$ such that the limit immersion $F_{\infty} : (M_{\infty}, g_{\infty}, p_{\infty}) \to (\mathbb{R}^N, g_{\mathrm{flat}}, 0)$ of Lemma \ref{step4} can be maximally extended to a codimension one solution of mean curvature flow $F_{\infty} : M_{\infty} \times (\tau^\ast, 0] \to \mathbb{R}^{N}$ with bounded curvature, and $F_{\infty}(\cdot, t)$ satisfies the same estimates and canonical neighborhood property as $F_{\infty}(\cdot, 0)$.

Moreover, the immersions $\tilde{F}_j$ converge smoothly to $F_{\infty}$ on compact subsets of $M_{\infty} \times (\tau^\ast, 0]$. 
\end{lemma}

\begin{proof}
By Lemma \ref{step4}, the mean curvature of the immersion $F_{\infty}$ is bounded from above by a constant $\Lambda > 8$. Since the sequence of immersions $\tilde F_j(\cdot, 0)$ converges locally smoothly to $F_{\infty}$, for every $\rho > 1$, we have 
\[
\limsup_{j \to \infty} \sup_{p \in B_{g(t_j)}(p_j, \rho Q_j^{-1})} Q_j^{-1} H(p, t_j) \leq 2\Lambda. 
\]
The short-range curvature estimates of Lemma \ref{step2} then imply that for every $\rho > 1$,
\[
\limsup_{j \to \infty} \sup_{(p,t) \in \hat P(p_j, t_j, \rho, \frac{1}{32 \eta \Lambda^2})} Q_j^{-1}H(p, t) \leq 8 \Lambda, 
\]
where recall $\hat P(p_j, t_j, \rho, \frac{1}{32 \eta \Lambda^2}) = B_{g(t_j)}(p_j, \rho Q_j^{-1}) \times [t_j - \frac{1}{32\eta\Lambda^2}Q_j^{-2}, t_j]$. 

If we take $\tau_1 := -\frac{1}{64 \eta \Lambda^2}$ and $j$ sufficiently large, then for every $\rho > 1$, we have uniform estimates for the second fundamental form of $\tilde F_j$ on the parabolic neighborhood $P(p_j, 0, \rho, \tau_1)$. Note these uniform estimates are independent of $\rho$. By Corollary \ref{flow_compactness}, a subsequence of these flows converges to a complete solution of mean curvature flow $F_{\infty}(\cdot, t)$ defined for $t \in [\tau_1, 0]$ with $F_{\infty}(\cdot, 0) = F_{\infty}$. Moreover, the limiting solution satisfies 
\[
\Lambda_1:= \sup_{(p,t) \in M_{\infty} \times [\tau_1, 0]} H(p, t) \leq 8 \Lambda.
\]
Now set $\tau_2 := \tau_1 - \frac{1}{64 \eta \Lambda_1^2}$. Using the short-range curvature estimates once more and passing to a further subsequence, we can extend the solution $F_{\infty}(\cdot, t)$ to the interval $t \in [\tau_2, 0]$ and the solution will satisfy $\Lambda_2 := \sup_{(p,t) \in M_{\infty} \times [\tau_2, 0]} H(p, t) \leq 8 \Lambda_1$. Analogously, for each $m \geq 1$, with $\tau_{m +1} = \tau_m - \frac{1}{64\eta \Lambda_m^2}$ and $\Lambda_{m+1} := \sup_{(p,t) \in M_{\infty} \times [\tau_{m+1}, 0]} H(p, t) \leq 2 \Lambda_m$, we get a complete solution of mean curvature flow $F_{\infty} : M_{\infty} \times [\tau_m, 0] \to \mathbb{R}^N$. 

Let $\tau^\ast := \lim_{m \to \infty} \tau_m$. Taking the limit along a suitable diagonal sequence of the flows $\tilde F_j$, we obtain a complete solution of mean curvature flow $F_{\infty} : M_{\infty} \times (\tau^\ast, 0] \to \mathbb{R}^N$. By construction, the solution has bounded mean curvature for each $t \in (\tau^\ast, 0]$. The solution satisfies the estimates $|\nabla A_{\infty}| \leq \eta H_{\infty}^2$, $|\nabla^2 A| \leq \eta H_{\infty}^3$, and $|A_{\infty}|^2 \leq \frac{1}{n-1} H_{\infty}^2$. By the planarity estimate, the solution is contained in an affine $(n+1)$-dimensional subspace of $\mathbb{R}^{N}$, so without loss of generality we may write $F_{\infty} : M_{\infty} \times (\tau^\ast, 0] \to \mathbb{R}^{n+1}$. Finally, by property (iii) together with Proposition \ref{model_nbhds}, any spacetime point $(p, t) \in M_{\infty} \times (\tau^\ast, 0]$ where $H_{\infty}(p, t) \geq 8$ has a canonical neighborhood which is either a $2\varepsilon$-neck, a $2\varepsilon$-cap, or a closed manifold diffeomorphic to $S^n$.   
\end{proof}

In this final lemma, we show that $\tau^\ast = -\infty$ using Hamilton's Harnack inequality for mean curvature flow. 

\begin{lemma}[showing the limit must be ancient]\label{step6}
If the flow $F_{\infty} : M_{\infty} \times (\tau^\ast, 0] \to \mathbb{R}^{N}$ constructed in Lemma \ref{step5} is extended maximally backwards in time, then $\tau^\ast = -\infty$.
\end{lemma}

\begin{proof}
Suppose $\tau^\ast > - \infty$. Recall $\tau_m$ and $\Lambda_m$ from the previous proof. This implies $\lim_{m \to \infty} (\tau_{m+1} - \tau_m) \to 0$ and consequently $\lim_{m \to \infty} \Lambda_m = \infty$. Thus, as we go backward in time to $\tau^\ast$, the mean curvature of $F_{\infty}(\cdot, t)$ blows up.

By the Harnack inequality,  $(t-\tau^\ast)^{\frac{1}{2}}H_{\infty}(p,t)$ is nondecreasing for each $p \in M_\infty$. Since $H_{\infty}(p, 0) \leq \Lambda$, this gives  
\[
H_{\infty}(p, t) \leq \sqrt{\frac{-\tau^\ast}{t - \tau^\ast}} \Lambda
\]
for all $t \in (\tau^\ast, 0]$ and $p \in M_\infty$. Applying Lemma \ref{dist_est}, we get
\[
0 \leq - \frac{d}{dt} d_{g_{\infty}(t)}(p, q) \leq C(n) \sqrt{\frac{-\tau^\ast}{t - \tau^\ast}} \Lambda
\]
for all $t \in (\tau^\ast, 0]$ and $p, q \in M_{\infty}$. The key point is that the right-hand side of the inequality above is integrable in $t$.  Integrating this inequality gives 
\[
d_{g_{\infty}(0)}(p, q) \leq d_{g_{\infty}(t)}(p,q) \leq d_{g_{\infty}(0)}(p, q) + C(n)(-\tau^\ast)\Lambda
\]
for all times $t \in (\tau^\ast, 0]$ and all points $p, q \in M_{\infty}$. 

By our rescaling procedure, clearly $H_{\infty}(p_{\infty}, 0) = 1$. The maximal principle implies the infimum of the mean curvature is nondecreasing and hence
\[
\inf_{p \in M_{\infty}} H_{\infty}(p, t) \leq \inf_{p \in M_{\infty}} H_{\infty}(p, 0) \leq 1
\]
for all $t \in (\tau^\ast, 0]$. It follows that we can find a point $q_{\infty} \in M_{\infty}$ where $H_{\infty}(q_{\infty}, t) \leq 2$ for $t = \tau^\ast + \frac{1}{64\eta}$. By the short-range curvature estimates of Lemma \ref{step2}, $H_{\infty}(q_{\infty}, t) \leq 8$ for $t \in (\tau^\ast, \tau^\ast + \frac{1}{64\eta}]$. In particular, if $m$ is sufficiently large, this implies $H_{\infty}(q_{\infty}, \tau_m) \leq 8$.  We claim
\[
\limsup_{m \to \infty} \sup_{p \in B_{g_{\infty}(\tau_m)}(q_{\infty}, \rho)} H_{\infty}(p, \tau_m) < \infty
\] 
for every $\rho > 1$. This follows from the argument used to prove Proposition \ref{step3}. For the proof there to work, we only need the pointwise derivative estimates and condition (iii). Both of these properties are satisfied by the limit. Since the mean curvature is bounded at bounded distance, the sequence of immersions $\hat F_m : (M_{\infty}, g_{\infty}(\tau_m), q_{\infty}) \to (\mathbb{R}^{n+1}, g_{\mathrm{flat}}, 0)$, where $\hat F_m(p) = F_{\infty}(p, \tau_m) - F_{\infty}(q_{\infty}, \tau_m)$, subsequentially converges to a smooth limit. By the argument in Lemma \ref{step4}, this limit has bounded curvature. It follows that there exists a constant $\Lambda_\ast > \Lambda$, independent of $\rho$, such that 
\[
\liminf_{m \to \infty} \sup_{p \in B_{g_{\infty}(\tau_m)}(q_{\infty}, \rho)} H_{\infty}(p, \tau_m) \leq \Lambda_\ast, 
\]
for every $\rho > 1$. The geodesic balls $B_{g_{\infty}(\tau_m)}(q_{\infty}, \rho)$ may change in size as $\tau_m \to \tau^\ast$. By our distance estimate, however, if $\rho$ is sufficiently large, then
\[
B_{g_{\infty}(0)}(q_{\infty}, \rho - C(n)(-\tau^\ast)\Lambda) \subset B_{g_{\infty}(\tau_m)}(q_{\infty}, \rho). 
\]
Consequently, 
\[
\liminf_{m \to \infty} \sup_{p \in B_{g_{\infty}(0)}(q_{\infty}, \rho)} H_{\infty}(p, \tau_m) \leq \Lambda_\ast, 
\]
for every $\rho > 1$. 

In summary, for every $\rho > 1$, we can find a large integer $m$ such that 
\[
 \sup_{p \in B_{g_{\infty}(0)}(q_{\infty}, \rho)} H_{\infty}(p, \tau_m) \leq 2\Lambda_\ast. 
\]
By the short-range curvature estimates (both forwards and backwards in time), taking $m$ sufficiently large, this implies 
\[
\sup_{t \in (\tau^\ast, \tau^\ast + \frac{1}{16\eta\Lambda_\ast^2}]} \sup_{p \in B_{g_{\infty}(0)}(q_{\infty}, \rho)} H_{\infty}(p, t) \leq 8\Lambda_\ast. 
\]
Since $\Lambda^\ast$ is independent of $\rho$, this gives
\[
\sup_{t \in (\tau^\ast, \tau^\ast + \frac{1}{16\eta\Lambda_\ast^2}]} \sup_{p \in M_{\infty}} H_{\infty}(p, t) \leq 8\Lambda_\ast. 
\]
This contradicts our observation that $\lim_{m \to \infty} \Lambda_m = \infty$. Therefore, we must have $\tau^\ast = -\infty$. 
\end{proof}

From the derivative estimates, the planarity estimate, the cylindrical estimate, and the main result of \cite{Naf19b}, the solution $F_{\infty} : M_{\infty} \times (-\infty, 0] \to \mathbb{R}^N$ is an ancient model solution. 

In conclusion, a subsequence of flows obtained by rescaling the solution $F$ around the points $(p_j, t_j)$ by $H(p_j, t_j)$ must converge to an ancient model solution. This, of course, contradicts property (ii), and thereby completes the proof of Theorem \ref{canonical-nbhd-thm}.


\section{Appendix: Local Compactness for Solutions of Mean Curvature Flow}

This appendix is meant to serve as a reference for the reader on a local compactness result for mean curvature flow given uniform local curvature estimates. This type of result is well-known and often used by experts. Our goal is to make precise the local compactness results we use in our proof of the canonical neighborhood theorem, as well as provide sufficiently many details for the proofs of these results. 

In \cite{Breu15}, building on work of Langer in \cite{Lan85}, Breuning proves a local compactness result for immersions that is close to what we desire here. The key idea behind the result, originally due to Langer, is that estimates for the second fundamental form, and its derivatives, yield estimates for the immersion and its derivatives when the immersion is given as a graph. The result of \cite{Breu15}, however, requires  estimates for the second fundamental form in extrinsic Euclidean balls rather than in intrinsic geodesic balls, which are more natural here. The result we are after is therefore closer in spirit Hamilton's compactness result \cite{Ham95b} for the Ricci flow and its local adaptation by Cao and Zhu in \cite{CaZ06}. We found their work to be a good reference for this type of result and our compactness theorem will be a corollary of theirs. 

First, a word on notation. Suppose that $F : (M,g) \to (\mathbb{R}^N, g_{\mathrm{flat}})$ is an isometric immersion. As usual, we let $\langle \cdot\,,  \cdot \rangle$ also denote that flat metric on $\mathbb{R}^N$. Let $\bar g$ be any other metric on $M$ and let $\nabla_{\bar g}$ denote the corresponding covariant derivative. To estimate derivatives of $F$ with respect to $\bar g$, we will think of the immersion $F$ as a tuple of $\mathbb{R}$-valued functions $F = (F^1, \dots, F^N)$ and use the notation 
\[
|\nabla^k_{\bar g} F|^2_{\bar g} := \sum_{m = 1}^N |\nabla^k_{\bar g} F^m|^2_{\bar g}
\]
where $|\nabla^k_{\bar g} F^m|_{\bar g}$ denotes the usual norm of the $(0, k)$-tensor $\nabla_{\bar g}^k F^m$ for each $m$. We will similarly view the second fundamental form $A$ as an $\mathbb{R}^N$-valued $(0, 2)$-tensor and the mean curvature vector $\vec H$ as an $\mathbb{R}^N$-valued function. This means we view $A$ as a section of $T^{\ast} M^{\otimes 2} \otimes F^\ast \mathbb{R}^N$ rather than as a section of $T^{\ast} M^{\otimes 2} \otimes NM$. Let $\nabla = \nabla_g$ and $\nabla^\perp = \nabla^\perp_g $ denote the induced connections on $T^\ast M^{\otimes k} \otimes F^\ast \mathbb{R}^N$ and $T^\ast M^{\otimes k} \otimes NM$ respectively. Acting on the second fundamental form, these derivatives are related in local coordinates $x_1, \dots, x_n$ on $M$ by 
\begin{align*}
\nabla_i A_{jk} &=  \frac{\partial}{\partial x_i} A_{jk}- \Gamma_{ij}^l A_{lk} - \Gamma_{ik}^l A_{jl}; \\
\nabla_i^\perp A_{jk} &= (\nabla_i A_{jk})^\perp = \nabla_i A_{jk} - \sum_{l =1}^n \langle A_{jk}, A_{il} \rangle e_l, 
\end{align*}
where $e_1, \dots, e_n$ is a local orthonormal frame for $TM$ (the equalities above hold as sections of $F^\ast \mathbb{R}^N = TM \oplus NM$). In particular, this means that $|\nabla A|^2 = |\nabla^\perp A|^2 + |\langle A, A \rangle|^2$ (where $\langle A, A \rangle$ is a $(0, 4)$-tensor). It is perhaps more natural to assume bounds on $|(\nabla^\perp)^k A|$ than it is to assume bounds on $|\nabla^k A|$. However, taking derivatives of the second identity above, we can estimate 
\[
|\nabla^k A| \leq |(\nabla^\perp)^k A| + C_k \sum_{p + q = k -1} |(\nabla^\perp)^p A||(\nabla^\perp)^q A| + C_k \sum_{p + q + r = k -1}  |(\nabla^\perp)^p A| |(\nabla^\perp)^q A| |(\nabla^\perp)^r A|.
\]
for a constant $C_k = C_k(n,N)$. Thus, estimates for $|\nabla^k A|$ for $0 \leq k \leq \bar k$ are equivalent to estimates for $|(\nabla^\perp)^k A|$ for $0 \leq k \leq \bar k$ up to some constants. In the lemmas and propositions below, we will assume estimates on $|\nabla^k A|$ rather than on $|(\nabla^\perp)^k A|$. In part, this is because $\nabla^{k +2} F = \nabla^k A$ as sections of $T^\ast M^{\otimes k +2} \otimes F^\ast \mathbb{R}^N$. However, we note that uniform bounds for $|\nabla^k F| = |\nabla_{F^\ast g_{\mathrm{flat}}}^k F|_{F^\ast g_{\mathrm{flat}}}$ (a quantity that is invariant under reparametrization) along a sequence of immersions is not enough to deduce compactness via Arzela-Ascoli. We must obtain estimates with respect to a fixed background metric.

We will reserve the notation $|\nabla^k A| = |\nabla_g^k A|_{g}$ to denote the norm of derivatives of the curvature with respect to the metric $g = F^\ast g_{\mathrm{flat}}$ induced by the immersion. For any other fixed background metric $\bar g$, we use the notation $|\nabla^k_{\bar g} A |_{\bar g}$. Estimates with respect to a fixed background metric are needed for extracting limits, but estimates for the second fundamental form are usually obtained with respect to the metric induced by the immersion.

\begin{definition}[cf. Definition 4.1.1 in \cite{CaZ06}]
Let $(M_j, g_j, p_j)$ be a sequence of pointed, complete $n$-dimensional Riemannian manifolds and suppose for each $j$ that $F_j :  (M_j, g_j, p_j) \to (\mathbb{R}^{N}, g_{\mathrm{flat}}, 0)$ is a pointed isometric immersion.  Let $B_j:= B_{g_j}(p_j, \rho_j) \subset M_j$ denote the open geodesic ball centered at $p_j \in M_j$ of radius $\rho_j \in (0,   \infty]$. Suppose $\rho_j \to \rho_\ast \in (0, \infty]$. Let $(B_{\infty}, g_{\infty}, p_{\infty})$ be a (possibly incomplete) Riemannian manifold such that $B_{\infty} = B_{g_{\infty}}(p_{\infty}, \rho_\ast)$, the open geodesic ball centered at $p_{\infty}$ of radius $\rho_\ast$ with respect to $g_{\infty}$. We say the sequence of pointed immersions $F_j|_{B_j}$ converges in $C^{\infty}_{\mathrm{loc}}$ to a pointed isometric immersion $F_{\infty} : (B_{\infty}, g_{\infty}, p_{\infty}) \to (\mathbb{R}^{N}, g_{\mathrm{flat}}, 0)$ if:
\begin{enumerate}
\item[$\bullet$] We can find a sequence of smooth, relatively compact open sets $U_j$ in $B_{\infty}$ satisfying $p_{\infty} \in U_j$, $U_j \subset U_{j+1}$, and $\cup_{j} U_j = B_{\infty}$.
\item[$\bullet$] We can find a sequence of diffeomorphisms $\phi_j : U_j \to \phi_j(U_j) \subset B_j$ satisfying $\phi_j(p_{\infty}) = p_j$. Moreover, for all $\rho \in (0, \rho_\ast)$, if $j$ is sufficiently large, then $B_{g_j}(p_j, \rho) \subset \phi_j(U_j)$. 
\item[$\bullet$] The sequence of immersions $F_j \circ \phi_j$ converges to $F_{\infty}$ smoothly with respect to $g_{\infty}$ on every compact subset of $B_{\infty}$. 
\end{enumerate}
\end{definition} 

To be specific, in the definition above, the sequence $F_j \circ \phi_j : U_j \to \mathbb{R}^N$ converges to $F_{\infty} : B_{\infty} \to \mathbb{R}^N$ smoothly with respect to $g_{\infty}$ on compact subsets as $\mathbb{R}^N$-valued functions. This means that for every nonnegative integer $\bar k$, every compact subset $K \subset B_{\infty}$, and every positive real number $\varepsilon > 0$, there exists $j_0 := j_0(\bar k, K, \varepsilon)$ such that $K \subset U_{j_0}$ and, for $j \geq j_0$, 
\[
\sup_K \sum_{k =0}^{\bar k} \big|\nabla_{g_{\infty}}^k\big((F_j \circ \phi_j) - F_{\infty}\big) \big|_{g_{\infty}}^2 < \varepsilon^2. 
\]
In particular, the pullback metrics $\phi_j^\ast g_j = (F_j \circ \phi_j)^\ast g_{\mathrm{flat}}$ converge smoothly on compact subsets of $B_{\infty}$ to $g_{\infty} = F_{\infty}^\ast g_{\mathrm{flat}}$. If $\rho_\ast = \infty$, then this means the sequence $(M_j, g_j, p_j)$ converges in the traditional pointed Cheeger-Gromov sense to a complete Riemannian manifold $(M_{\infty}, g_{\infty}, p_{\infty})$. 

The following proposition is the analogue of Theorem 4.1.2 in \cite{CaZ06}.

\begin{proposition}[Local compactness of pointed immersions]\label{immersion_compactness}
Let $(M_j, g_j, p_j)$ be a sequence of pointed, complete $n$-dimensional Riemannian manifolds and suppose for each $j$ that $F_j : (M_j, g_j, p_j) \to (\mathbb{R}^{N}, g_{\mathrm{flat}}, 0)$ is a pointed isometric immersion. Consider a sequence of radii $\rho_j \in (0, \infty]$ such that $\rho_j \to \rho_\ast \in (0, \infty]$ and let $B_j := B_{g_j}(p_j, \rho_j)$. Suppose that for every radius $0 < \rho < \rho_\ast$ and every integer $k \geq 0$, there exists a constant $\Lambda_k(\rho)$, independent of $j$, and a positive integer $j_0(k, \rho)$ such that for every $j \geq j_0(k, \rho)$ the $k$\textsuperscript{th} covariant derivative of the second fundamental form $A_j$ of the immersion $F_j$ satisfies the pointwise estimate
\[
\sup_{B_{g_j}(p_j, \rho)} |\nabla^k A_j|_{g_j} \leq \Lambda_k(\rho). 
\]
Then there exists a subsequence of the immersions $F_j|_{B_j}$ which converges in $C^{\infty}_{\mathrm{loc}}$ to a pointed isometric immersion $F_{\infty} : (B_{\infty}, g_{\infty}, p_{\infty}) \to (\mathbb{R}^{N}, g_{\mathrm{flat}}, 0)$ of an open geodesic ball $B_{\infty} = B_{g_{\infty}}(p_{\infty}, \rho_\ast)$. If $\rho_\ast = \infty$, then the limiting Riemannian manifold is complete.
\end{proposition}
 
Before we discuss the proof of the proposition, let us show how to deduce a local convergence result for mean curvature flow as a corollary. We will take the following definition for local convergence of flows. 

\begin{definition}[cf. Definition 4.1.3 in \cite{CaZ06}]
Fix $\tau < 0$. Let $(M_j, g_j(t), p_j)$ be a sequence of evolving, pointed, complete $n$-dimensional Riemannian manifolds for $t \in (\tau, 0]$. Suppose $F_j(\cdot, t): (M_j, g_j(t), p_j) \to \mathbb{R}^N$, for $t \in (\tau, 0]$, is a sequence of smoothly evolving immersions satisfying $F_j(p_j, 0) = 0$. Consider a sequence of radii $\rho_j \in (0, \infty]$ such that $\rho_j \to \rho_\ast \in (0, \infty]$ and let $P_j:= B_{g_j(0)}(p_j, \rho_j) \times (-\tau, 0]$. We say the sequence of pointed evolving immersions $F_j|_{P_j}$ converges in $C^{\infty}_{\mathrm{loc}}$ to a pointed evolving immersion $F_{\infty}(\cdot, t) : (B_{\infty}, g_{\infty}(t), p_{\infty}) \to \mathbb{R}^N$ for $t \in (\tau, 0]$ of an evolving, pointed, $n$-dimensional Riemannian manifold $B_{\infty} = B_{g_{\infty}(0)}(p_{\infty}, \rho_\ast)$, if: 
\begin{enumerate}
\item[$\bullet$] We can find a sequence of (time-independent) smooth, relatively compact open sets $U_j$ in $B_{\infty}$ satisfying $p_{\infty} \in U_j$, $U_j \subset U_{j+1}$, and $\cup_{j} U_j = B_{\infty}$.
\item[$\bullet$] We can find a sequence of (time-independent) diffeomorphisms $\phi_j : U_j \to \phi_j(U_j) \subset B_{g_j}(p_j, \rho_j)$ satisfying $\phi_j(p_{\infty}) = p_j$. Moreover, for all $\rho \in (0, \rho_\ast)$, if $j$ is sufficiently large, then $B_{g_j}(p_j, \rho) \subset \phi_j(U_j)$.
\item[$\bullet$] The sequence of evolving immersions $F_j(\cdot, t) \circ \phi_j$ converges to $F_{\infty}(\cdot, t)$ smoothly with respect to $g_{\infty}(0)$ on every compact subset of $B_{\infty} \times (\tau, 0]$. 
\end{enumerate}
\end{definition}

To be specific, in the definition above, the sequence $F_j(\cdot, t) \circ \phi_j : U_j \to \mathbb{R}^N$ converges to $F_{\infty} : B_{\infty} \times (\tau, 0] \to \mathbb{R}^N$ smoothly with respect to $g_{\infty}(0)$ on compact subsets if for every nonnegative integer $k$, compact subset $K \subset B_{\infty}$, compact subset $[\tilde \tau, 0] \subset (\tau, 0]$, and positive real number $\varepsilon > 0$, there exists $j_0 := j_0(k, K, \tilde \tau, \varepsilon)$ such that $K \subset U_{j_0}$ and, for $j \geq j_0$, 
\[
\sup_{t \in [\tilde{\tau}, 0]} \sup_{K} \sum_{m =0}^{k} \big|\nabla_{g_{\infty}(0)}^m\big(F_j(\cdot, t) \circ \phi_j- F_{\infty}(\cdot, t)\big) \big|_{g_{\infty}(0)}^2 < \varepsilon^2. 
\]
 
Here is the compactness result for local solutions of mean curvature flow that we are after. This is the analogue of Theorem 4.1.5 in \cite{CaZ06}.

\begin{corollary}[Local compactness of mean curvature flow]\label{flow_compactness}
Fix $\tau < 0$. Let $(M_j, g_j(t), p_j)$ be a sequence of evolving, pointed, complete $n$-dimensional Riemannian manifolds for $t \in (\tau, 0]$. Suppose $F_j(\cdot, t): (M_j, g_j(t), p_j) \to \mathbb{R}^N$, for $t \in (\tau, 0]$, is a sequence of smoothly evolving immersions satisfying $F_j(p_j, 0) = 0$. Consider a sequence of radii $\rho_j \in (0, \infty]$ such that $\rho_j \to \rho_\ast \in (0, \infty]$. Suppose $F_j$ is a solution to mean curvature flow on the parabolic neighborhood $P_j:= B_{g_j(0)}(p_j, \rho_j) \times (\tau, 0]$.  Finally, suppose that for every radius $0 < \rho < \rho_\ast$ there exists a constant $\Lambda(\rho)$, independent of $j$, and a positive integer $j_0(\rho)$ such that for every $j \geq j_0(\rho)$, the second fundamental form $A_j$ of the evolving immersion $F_j(\cdot, t)$ satisfies the pointwise estimate 
\[
\sup_{B_{g_j(0)}(p_j, \rho) \times (\tau, 0]} |A_j|_{g_j} \leq \Lambda(\rho). 
\]
Then there exists a subsequence of solutions $F_j$ such that $F_j|_{P_j}$ converge in $C^{\infty}_{\mathrm{loc}}$ to a solution of mean curvature flow $F_{\infty} : B_{\infty} \times (\tau, 0] \to \mathbb{R}^N$ with $g_{\infty}(t) = F_{\infty}(\cdot, t)^\ast g_{\mathrm{flat}}$ and $B_{\infty} = B_{g_{\infty}(0)}(p_{\infty}, \rho_\ast)$. If $\rho_\ast = \infty$, the limiting solution is complete at time $t = 0$.
\end{corollary}
Note that if $\rho_\ast = \infty$ and if the bounds for the second fundamental form in the proposition above can be taken to be independent of $\rho$, then the solution will be complete on every time-slice. 
 
It is straightforward to see that the assumptions of the corollary together with Proposition \ref{immersion_compactness} allow us to extract a limiting immersion at the time $t = 0$. In order to extend the convergence backwards in time, we will use the following lemma, which is an adaptation of Lemma 4.1.4 in \cite{CaZ06} to our setting.  

\begin{lemma}[cf. Lemma 4.1.4 in \cite{CaZ06}]\label{time_extension}
Let $(B, g, p)$ be a pointed Riemannian manifold, $K \subset B$ be a compact subset, and $\tilde F_j(\cdot, t)$ be a sequence of pointed (i.e, $\tilde F_j(p, 0) = 0$) solutions of mean curvature flow defined on $K \times [\tilde \tau, 0]$. Let $\tilde g_j(t) = \tilde F_j(\cdot, t)^\ast g_{\mathrm{flat}}$. Let $\nabla$ denote covariant derivative of $g$ and $\tilde \nabla = \tilde \nabla_{\tilde g_j}$ denote the covariant derivative of $\tilde g_j(t)$ for each $j$. Suppose there exist constants $C_k$ (independent of $j$) for each integer $k \geq 0$ such that  
\begin{enumerate}
\item[(a)] $C_0^{-1} g \leq \tilde g_j(0) \leq C_0 g$ on $K$ for all $j$;
\item[(b)] for each $k \geq 0$, $|\nabla^k\tilde F_j(\cdot, 0) |_{g} \leq C_k$ on $K$ for all $j$; 
\item[(c)] $|\nabla_{\tilde g_j}^k \tilde A_j |_{\tilde g_j} \leq C_k$ on $K \times [\tilde \tau, 0]$ for all $j$. 
\end{enumerate}
Then there exists constants $\tilde C_k$ (independent of $j$) for each integer $k \geq 0$ such that 
\[
\tilde C_0^{-1} g \leq \tilde g_j(t) \leq \tilde C_0 g; \qquad |\nabla^k \tilde F_j  |_{g} \leq \tilde C_k \;\; (k\geq 0)
\]
 on $K \times [\tilde \tau, 0]$ for all $j$. 
\end{lemma}
\begin{proof}
The proof of this lemma differs very little from the proof of Lemma 4.1.4 in \cite{CaZ06}. This lemma also follows from the results in Appendix A of Brendle's book \cite{Bre10}. So we will just highlight the key points. Let $p \in K$, $v \in T_pB$, and define $\mu(t) = \tilde g_j(t)(v, v)$. Since   
\[
\Big|\frac{\partial}{\partial t} \tilde g_j(v, v)\Big|= 2 \Big|\langle \tilde A_j(v, v), \tilde {\vec{H}}_j \rangle\Big| \leq C(n)C_0^2 \tilde g_j(v, v), 
\]
(recall $\langle \cdot\,, \cdot \rangle = g_{\mathrm{flat}}$) we have $|\mu'(t)| \leq C\mu(t)$. From this and (a), the uniform equivalence of the metrics $\tilde C_0^{-1} g \leq \tilde g_j(t) \leq \tilde C_0 g$ readily follows. Let $\tilde \Gamma_j$ and $\Gamma$ denote the connection coefficients of $\tilde g_j$ and $g$ respectively. The expression $\frac{\partial}{\partial t} (\tilde \Gamma_j - \Gamma) = \frac{\partial}{\partial t} \tilde \Gamma_j$ is tensorial and, since $\frac{\partial}{\partial t} \tilde \Gamma_j = \tilde g_j^{-1} \ast \tilde \nabla_{\tilde g_j} (\frac{\partial}{\partial t} \tilde g_j)$, assumption (c) implies an estimate $\big|\frac{\partial}{\partial t} (\tilde \Gamma_j - \Gamma) \big|_g \leq C$ on $K \times [\tilde \tau, 0]$. On the other hand, assumption (b) implies (for each $k$) that $|\nabla^k \tilde g_j(0)| \leq C$ on $K$, which implies $|\tilde \Gamma_j(0) - \Gamma|_g \leq C$. By integration one concludes $|\tilde \Gamma_j - \Gamma|_g \leq C$ on $K \times [\tilde \tau, 0]$. Now that we have estimates for the difference of the connection coefficients, we estimate 
\begin{align*} 
\Big|\frac{\partial}{\partial t} \nabla^k \tilde g_j\Big|_g =  \Big| \nabla^k \frac{\partial}{\partial t} \tilde g_j\Big|_g  \leq \Big|\tilde \nabla^k \frac{\partial}{\partial t} \tilde g_j\Big|_g  + \Big| (\nabla^k - \tilde \nabla^k) \frac{\partial}{\partial t}  \tilde g_j \Big|_g. 
\end{align*}
In light of the evolution equation for $\tilde g_j$ and assumption (c), we can bound $|\tilde \nabla^k \frac{\partial}{\partial t} \tilde g_j|_g \leq C$. By induction, we can bound the second term $|(\nabla^k - \tilde \nabla^k) \frac{\partial}{\partial t}  \tilde g_j |_g$ by $C + C|\nabla^k \tilde g_j|_g$. Hence, we have $\frac{\partial}{\partial t}  |\nabla^k \tilde g_j|_g \leq C + C|\nabla^k \tilde g_j|_g$ By integration and assumption (b), we obtain the estimate $|\nabla^k \tilde g_j|_g \leq C$ on $K \times [\tilde \tau, 0]$ for each $k$. See Lemma A.3 in \cite{Bre10} and the proof in \cite{CaZ06} for further details.  In a similar fashion, we have 
\begin{align*}
\Big|\frac{\partial}{\partial t}(\nabla^k \tilde F_j) \Big|_g &= \Big|\nabla^k \frac{\partial}{\partial t} \tilde F_j \Big|_g = \Big|\nabla^k \tilde{\vec H}_j \Big|_g  \leq \Big|\tilde \nabla^k  \tilde{\vec H}_j  \Big|_g + \Big|(\nabla^k - \tilde \nabla^k) \tilde{\vec H}_j  \Big|_g. 
\end{align*}
By assumption (c) the first term is bounded by a constant. Our estimates for the metric and its derivatives together with assumption (c) give control of the second term. For example, because $\tilde \nabla \tilde {\vec H}_j = \nabla \tilde {\vec H}_j$, we have
\begin{align*}
(\nabla^2 - \tilde \nabla^2)  \tilde{\vec H}_j & =(\nabla - \tilde \nabla)( \nabla \tilde{\vec H}_j) + \tilde \nabla (\nabla - \tilde \nabla) \tilde{\vec H}_j \\
& = (\nabla - \tilde \nabla) (\tilde \nabla \tilde{\vec H}_j)\\
& = (\Gamma - \tilde \Gamma_j) \ast \tilde \nabla \tilde{\vec H}_j
\end{align*}
Note that $\Gamma - \tilde \Gamma_j = \tilde g_j^{-1} \ast \nabla \tilde g_j$. For the general case, see Lemma A.4 in \cite{Bre10}. By integration and assumption (b), we obtain the desired estimates. 
\end{proof}

Now we can give a proof of Corollary \ref{flow_compactness}. 

\noindent \textbf{Proof of Corollary \ref{flow_compactness}.}
We have uniform estimates for $A_j$ on the parabolic neighborhood $B_{g_j(0)}(p_j, \rho) \times (\tau, 0]$ for each $0 < \rho < \rho_\ast$ (assuming $j \geq j_0(\rho)$). Therefore, by standard interior estimates for mean curvature flow, for each integer $k \geq 0$ there exists a constant $C_k(\rho)$ such that 
\[
\sup_{B_{g_j(0)}(p_j, \rho)} |\nabla_{g_j(0)}^k A_j(0)|_{g_j(0)} \leq C_k(\rho)
\]
for each $0 < \rho < \rho_\ast$ (assuming $j \geq j_0(\rho)$). In particular, the sequence $F_j(\cdot, 0) : (M_j, g_j(0), p_j) \to (\mathbb{R}^{N}, g_{\mathrm{flat}}, 0)$ satisfies the hypotheses of Proposition \ref{immersion_compactness}. Thus after passing to a subsequence, which we still denote by $F_j$, we can find a pointed Riemannian manifold $(B_\infty, g_{\infty}, p_{\infty})$, a sequence of domains $U_j$ exhausting $B_{\infty}$, and injective smooth maps $\phi_j : U_j \to B_{g_j(0)}(p_j, \rho_j)$ such that $\tilde F_j  := F_j \circ \phi_j$ converges smoothly on compact subsets of $B_{\infty}$ with respect to $g_{\infty}$ to a pointed isometric immersion $F_{\infty} : (B_{\infty}, g_{\infty}, p_{\infty}) \to (\mathbb{R}^N, g_{\mathrm{flat}}, 0)$. Here $B_{\infty}$ is the open geodesic ball $B_{g_{\infty}}(p_{\infty}, \rho_\ast)$.

Now we can apply Lemma \ref{time_extension}. Let $K \subset B_{\infty}$ be any compact subset and $[\tilde \tau, 0] \subset (\tau, 0]$. Let $\tilde g_j(t) := \tilde F_j(\cdot, t)^\ast g_{\mathrm{flat}}$. To simplify notation, let $F := F_{\infty}$ and $g := g_{\infty}$. Let $\tilde \nabla_{\tilde g_j}$ and $\nabla$ denote the covariant derivatives of $\tilde g_j$ and $g$ respectively. By the definition of convergence, after passing to a suitable diagonal subsequence, for every integer $k \geq 0$, there exists a constant $C_k$ such that 
\[
|\nabla^k \tilde F_j(\cdot, 0)|_g \leq  \big|\nabla^k\big(\tilde F_j(\cdot, 0) - F\big) \big|_g  + |\nabla^k F \big|_g \leq C_k
\]
on $K$ for all $j$. Clearly, we also have $C_0^{-1} g \leq \tilde g_j(0) \leq C_0 g$ on $K$ for all $j$. Finally, by diffeomorphism invariance and interior estimates we obtain $|\tilde\nabla_{\tilde g_j}^k \tilde A_j |_{\tilde g_j} = |\nabla_{g_j}^k A_j|_{g_j} \leq C_k$ on $K \times [\tilde \tau, 0]$ for all $j$. Thus, assumptions (a) - (c) of Lemma \ref{time_extension} are satisfied and consequently, we have uniform estimates $\tilde C_0^{-1} g \leq \tilde g_j(t) \leq \tilde C_0 g$ and $|\nabla^k \tilde F_j  |_{g} \leq \tilde C_k$ for $k\geq 0$ with respect to the fixed background metric $g$. We can now use Arzela-Ascoli with a standard diagonalization argument to extract a subsequence that converges uniformly on compact subsets of $B_{\infty} \times (\tau, 0]$. In the limit, we obtain a family of smooth maps $F_{\infty}(\cdot, t) : B_{\infty} \to \mathbb{R}^N$ for $t \in (\tau, 0]$. Evidently, $F_{\infty}(\cdot, 0) = F = F_{\infty}$ since we already have convergence at time $t = 0$. It remains to verify the family $F_{\infty}(\cdot, t)$ is a solution of mean curvature flow. The (0,2)-tensor $g_{\infty}(t) = F_{\infty}(\cdot , t)^\ast g_\mathrm{flat}$ is the limit of the nondegenerate metrics $\tilde g_j(t)$ each of which are uniformly equivalent to the Riemannian metric $g = g_{\infty}(0)$. In particular, $g_{\infty}(t)$ is itself a Riemannian metric and $F_{\infty}(\cdot, t)$ is a family of immersions. Finally, as a limit of solutions of mean curvature flow, it is clear that $F_{\infty}(\cdot, t)$ satisfies the same equation. \qed

In the remainder of this appendix, we will give a proof of Proposition \ref{immersion_compactness}. Our approach will be first to extract an intrinsic limit along a suitable subsequence, using the work of Hamilton and Cao-Zhu's localization of it. To do so, we will need an injectivity radius estimate, which we obtain from the following lemma.

\begin{lemma}\label{injectivity_radius}
Let $F : (M, g, p) \to (\mathbb{R}^N, g_{\mathrm{flat}}, 0)$ be a pointed isometric immersion of a smooth connected complete Riemannian manifold. Suppose that for some $\rho \in (0, \infty)$, 
\[
\sup_{B_g(p, \rho)} |A| \leq \Lambda.
\]
There exists a positive constant $\delta := \delta(n, \Lambda, \rho) > 0$ such that $\mathrm{inj}(M, p) \geq \delta$.
\end{lemma}
\begin{proof}
After composing $F$ with an isometry of $\mathbb{R}^N$, we may assume that $F(0) = 0$ and $dF_p(T_pM) = \mathbb{R}^n \times \{0\} \subset \mathbb{R}^N$. Let $\pi : \mathbb{R}^N \to \mathbb{R}^n$ denote the projection onto the first $n$-coordinates of $\mathbb{R}^N$. Let 
\[
r_0 := \frac{1}{500} \min\{\Lambda^{-1}, \rho \}.
\] 
We will first show that the immersion can be expressed as a graph over a ball in $dF_p(T_pM)$ of radius proportional to $r_0$. Then, in the graphical parametrization we can estimate the intrinsic volume, and from this the local injectivity radius estimate will follow. 

\underline{Step 1:} First, we show that $\pi \circ F$ is injective on $B_g(p, 15 r_0)$.  If not, then there exist distinct points $q_0, q_1 \in B_g(p, 15 r_0)$ such that $\pi \circ F(q_0) = \pi \circ F(q_1)$. Let $\gamma : [0, 1] \to M$ be a minimizing geodesic with $\gamma(0) = q_0$ and $\gamma(1) = q_1$. Now $d_{g}(q_0, q_1) < 30 r_0$, which implies 
\[
d_g(p, \gamma(t)) \leq d_g(p, q_0) + d_g(q_0, \gamma(t)) < 45r_0 < \rho.
\]
So $\gamma([0,1]) \subset B_g(p,\rho)$ and therefore $|A|(\gamma(t)) \leq \Lambda$ for each $t \in [0,1]$. Let $\tilde \gamma(t) = F \circ \gamma(t)$ and $\hat \gamma(t) = \pi \circ \tilde \gamma(t)$. Because $\gamma(t)$ is a geodesic, we have $\tilde \gamma'' = A(\gamma', \gamma')$ and $|\tilde \gamma'| = |\gamma'| = d_g(q_0, q_1)$. We claim that $|\hat \gamma'(t)|^2 \geq \frac{1}{10} d_g(q_0, q_1)^2$ for all $t \in [0, 1]$. To see this fix some $t \in [0, 1]$ and let $X(s)$ be the parallel transport of $\gamma'(t)$ along a minimizing geodesic $\sigma(s)$ connecting $\sigma(0) = p$ to $\sigma(1) = \gamma(t)$. Note $|\sigma'| \leq 45r_0$. Let $\tilde X(s) = dF_{\sigma(s)}(X(s))$ and note that $|\pi(\tilde X(0))| = |\tilde X(0)| = |X(0)| = d_g(q_0, q_1)$. Now 
\[
\Big|\frac{d}{ds} \frac{1}{2}|\pi(\tilde X)|^2 \Big| = \big|\langle\pi( A(\sigma', X)), \pi(\tilde X) \rangle\big| \leq |A||\sigma'||X|^2 \leq 45 \Lambda r_0 d_g(q_0, q_1)^2  \leq \frac{9}{20} d_{g}(q_0, q_1)^2. 
\]
Therefore, 
\[
\frac{1}{2} |\hat \gamma'(t)|^2 = \frac{1}{2} |\pi(\tilde X(1))|^2 \geq \frac{1}{2} d_g(q_0, q_1)^2 - \frac{9}{20} d_g(q_0, q_1)^2 \geq \frac{1}{20} d_g(q_0, q_1)^2, 
\]
which implies $|\hat \gamma'(t)|^2 \geq \frac{1}{10} d_{g}(q_0, q_1)^2$. Now we consider the function $f(t) = \frac{1}{2}|\hat \gamma(t) - \hat \gamma(0)|^2$. Then $f' = \langle \hat \gamma', \hat \gamma - \hat \gamma(0) \rangle$ and $f'' = \langle \hat \gamma'', \hat \gamma - \hat \gamma(0) \rangle + |\hat \gamma'|^2$. Also, $f(0) = f(1) = 0$, $f'(0) = 0$ and $f''(0) = |\hat\gamma'(0)|^2 \geq \frac{1}{10} d_g(q_0, q_1) > 0$. Consequently, $f$ attains its maximum at some point $t_0 \in (0, 1)$. Since $|\hat \gamma''| = |\pi( A(\gamma', \gamma'))| \leq \Lambda d_g(q_0, q_1)^2$ and $|\hat \gamma - \hat \gamma(0)| \leq 30 r_0$, at the point $t_0$, we obtain the inequality 
\[
0 \geq f''(t_0) \geq - 30r_0 \Lambda d_g(q_0, q_1)^2 + \frac{1}{10} d_g(q_0, q_1)^2. 
\]
But this implies that $r_0 \geq \frac{1}{300} \Lambda^{-1}$, in contradiction with its definition. 

\underline{Step 2:} Next, we show that $c(n) r_0^n \leq \mathrm{Vol}(B_g(p, 15r_0))\leq C(n) r_0^n$. Let $\Omega := \pi(F(B_g(p, 15r_0))) \subset \mathbb{R}^n$. Let $|| \cdot ||$ denote the Euclidean norm on $\mathbb{R}^n$ and $D$ its standard derivative. Now since $\pi$ is injective on $F(B_g(p, 15 r_0))$, we can find a smooth function $f : \Omega \to \mathbb{R}^{N-n}$ such that $f(0) = 0$ , $Df(0) = 0$, and $\mathrm{graph}(f) = F(B_g(p, 15r_0))$. For graphical parametrizations, it is straightforward to show (see Lemma 2.2 in \cite{Breu15}) the inequality
\[
||D^2 f|| \leq (1 + ||Df||^2)^{\frac{3}{2}} (|A|_g \circ F^{-1})
\]
holds on $\Omega$. 
Let $\tilde r > 0$ be the maximal radius such that $B^n_{\tilde r} \subset \Omega$, where $B^n_{\tilde r}$ denote the Euclidean $n$-ball of radius $r$ centered at the origin.  Clearly, $\tilde r \leq 15 r_0$. For $x \in B^n_{\tilde r} \setminus \{0\}$, write $x = r \omega$ where $\omega$ is a unit vector and $r \in (0, \tilde r)$. Consider the function $\mu(t) = ||Df||^2(t \omega)$ for $t \in [0, r]$. Noting that $\mathrm{graph}(f)\subset F(B_g(p, \rho))$, the inequality above gives
\[
\mu'(t) \leq 2||D^2f||(t\omega)\, \mu(t)^\frac{1}{2} \leq 2(1 + \mu(t))^{\frac{3}{2}} \mu(t)^{\frac{1}{2}} \Lambda. 
\]
We can rewrite this as 
\[
\frac{d}{dt} \Big(\frac{\mu(t)}{1 + \mu(t)}\Big)^{\frac{1}{2}} \leq \Lambda
\]
Since $\mu(0) = 0$, integrating from $0$ to $r$ gives 
\[
\Big(\frac{||Df||^2(x)}{1 + ||Df||^2(x)}\Big)^\frac{1}{2}  \leq r \Lambda \leq 15\, r_0 \Lambda < \frac{1}{20}. 
\]
Therefore, we conclude that $||Df|| \leq \frac{1}{100}$ on $B^n_{\tilde r}$. Finally, using this slope bound, we can derive the volume estimate. Consider a direction $\omega$ such that $\tilde r \omega \in \partial \Omega$. Using our estimate for the slope of $f$, the path $\tilde \sigma(t) = (t \omega, f(t\omega))$ clearly has length bounded by $\frac{3}{2}\tilde r$. Since the path $\sigma(t) := F^{-1} \circ \tilde \sigma(t)$ is a path in $M$ from $p$ to the boundary of $B_g(p, 15r_0)$ and $F$ preserves lengths, we conclude $\tilde r \geq 10r_0$. Since $B_{\tilde r}^n \subset \Omega$, using the graphical parametrization, we obtain $\mathrm{Vol}(B_g(p, 15r_0)) \geq c(n) r_0^n$. On the other hand, the slope bound for $f$ gives the reverse inequality $\mathrm{Vol}(B_g(p, 15r_0))\leq C(n) r_0^n$.

\underline{Conclusion:} Using the Gauss equation, we can bound the absolute value of the sectional curvature in $B_g(p, \rho)$ by $2\Lambda$. Using the volume estimates in Step 2, it follows from Theorem 4.2.2. in \cite{CaZ06} (which is a local injectivity radius estimate due to Cheeger-Gromov-Taylor in \cite{CGT82}) that $\mathrm{inj}(M, p) \geq c(n, \Lambda, \rho) r_0$. 
\end{proof}

We can now prove Proposition \ref{immersion_compactness}. 

\noindent \textbf{Proof of Proposition \ref{immersion_compactness}.} We will complete the proof in two steps.  

\underline{Step 1:} We will first take an intrinsic limit in the sense of Definition 4.1.1 in \cite{CaZ06} by applying Theorem 4.1.2. in \cite{CaZ06}. To that end, we consider the sequence of geodesic balls $B_j = B_{g_j}(p_j, \rho_j) \subset M_j$ and verify two conditions. 
\begin{enumerate}
\item[(a)] Consider a radius $\rho < \rho_\ast$ and an integer $k \geq 0$. Via the Gauss equation, estimates for covariant derivatives of the second fundamental form yield corresponding estimates for the Riemannian curvature tensor $\mathrm{Rm}(g_j)$ of the metric $g_j$. In particular, $|\nabla^k \mathrm{Rm}(g_j)|$ can be bounded pointwise by an expression in $|\nabla^l A_j|$ for $0 \leq l \leq k$. Thus we can find a constant $\tilde \Lambda_k(\rho)$, independent of $j$ and a positive integer $j_1(k, \rho)$ such that $j \geq j_1(k, \rho)$
\[
\sup_{B_{g_j}(p_j, \rho)} |\nabla^k \mathrm{Rm}(g_j)| \leq  \tilde \Lambda_k(\rho). 
\]
\item[(b)] Let $\tilde \rho = \min\{\frac{1}{2} \rho_\ast, 1\}$. After passing to a subsequence, we have $|A_j| \leq \Lambda_0(\tilde \rho)$ on the geodesic ball $B_{g_j}(p_j, \tilde \rho)$ for all $j$. Thus, by Lemma \ref{injectivity_radius}, there exists a positive constant $\delta:= \delta(n, \Lambda_0(\tilde \rho), \tilde \rho)$, independent of $j$, such that the injectivity radius of $M_j$ at $p_j$ in the metric $g_j$ satisfies 
\[
\mathrm{inj}(M_j, p_j) \geq \delta.
\]
\end{enumerate}
Having verified conditions (a) and (b), we may now apply Theorem 4.1.2. in \cite{CaZ06} to obtain the following conclusion: there exists a subsequence of pointed geodesic balls $(B_j, g_j, p_j)$ which converge to a pointed geodesic ball $(B_{\infty}, g_{\infty}, p_{\infty})$ centered at a point $p_{\infty}$ of radius $\rho_\ast$ (that is, $B_{\infty} = B_{g_{\infty}}(p_{\infty}, \rho_\ast)$) in the intrinsic $C^{\infty}_{\mathrm{loc}}$ topology. This means, we can find a sequence of exhausting open sets $U_j$ in $B_{\infty}$, each containing $p_{\infty}$, and a sequence of diffeomorphisms $\phi_j : U_j \to \phi_j(U_j) \subset B_j \subset M_j$ such that $\phi_j(p_{\infty}) = p_j$ and the metrics $\tilde g_j := \phi_j^\ast g_j$ converge to $g_{\infty}$ in the smooth topology on every compact subset of $B_{\infty}$. Moreover, the proof of Theorem 4.1.2. in \cite{CaZ06} implies $U_j \subset U_{j+1}$ and for every $\rho \in (0, \rho_\ast)$, if $j$ is sufficiently large then $B_{g_j}(p_j, \rho) \subset \phi_j(U_j)$.  

\underline{Step 2:} Now that we have an intrinsic limit, it is straightforward to show that a subsequence of the immersions $\tilde F_j := F_j \circ \phi_j : U_j \to \mathbb{R}^N$ converges to a limit $F_{\infty} : B_{\infty} \to \mathbb{R}^N$, smoothly with respect to $g_{\infty}$ on compact subsets of $B_{\infty}$. Consider any compact subset $K \subset B_{\infty}$. If $j$ is sufficiently large, then $K \subset U_j$ and $\tilde F_j$ is defined on $K$. Moreover, there exists $\rho := \rho(K) \in (0, \rho_\ast)$ such that $\phi_j(K) \subset B_{g_j}(p_j, \rho)$ if $j$ is large enough. By diffeomorphism invariance, $|\nabla^k_{\tilde g_j} \tilde F_j|_{\tilde g_j} = |\nabla^k_{g_j} F_j|_{g_j}$. Recall from the discussion at the beginning of the appendix that $\nabla^k_{g_j} F_j = \nabla^{k-2}_{g_j} A_j$.  Therefore, given $K$ and $k \geq 2$, our uniform bounds for the second fundamental form and its covariant derivatives imply uniform bounds for $|\nabla^k_{\tilde g_j} \tilde F_j|_{\tilde g_j}$ on $K$ once $j$ is sufficiently large. On the other hand, for $k =1$, we have $|\nabla_{\tilde g_j} \tilde F_j |_{\tilde g_j} = n$ since $\tilde g_j = \phi^\ast g_j = \tilde F_j^\ast g_{\mathrm{flat}}$. Thus, since the metrics $\tilde g_j$ converge to $g_{\infty}$ on $K$, given a compact $K \subset B_{\infty}$ and an integer $k \geq 1$, we obtain uniform estimates for $|\nabla^k_{g_{\infty}} \tilde F_j |_{g_{\infty}}$ if $j$ is sufficiently large. By assumption $\tilde F_j(p_{\infty}) = F_j(p_j) = 0$ and so with the first derivative estimate, we conclude $|\tilde F_j| \leq C$ on $K$ as well for $j$ large enough. By the classical Arzela-Ascoli and diagonalization, we can find a subsequence of the $\tilde F_j$  which converge smoothly with respect to $g_{\infty}$ on every compact subset of $B_{\infty}$ to smooth limit $F_{\infty} : B_{\infty} \to \mathbb{R}^N$. Since $g_{\infty} = \lim_{j \to \infty} \tilde g_j = \lim_{j \to \infty} \tilde F_j^\ast g_{\mathrm{flat}} = F_{\infty}^\ast g_{\mathrm{flat}}$, the limiting $\mathbb{R}^N$-valued function is an immersion, completing the proof. \qed

\sc{Department of Mathematics, Columbia University, New York, NY 10027}

\end{document}